\documentclass[a4paper]{article}
\usepackage{fullpage}
\usepackage[Symbolsmallscale]{upgreek}

\usepackage{enumitem}
\usepackage{calc}

\usepackage{lmodern} 

\usepackage{graphicx} 

\usepackage{amsmath, accents}
\usepackage{amssymb,tikz}
\usepackage{pict2e}
\usepackage{amsthm}
\usepackage{amscd}
\usepackage{mathrsfs}
\usepackage{todonotes}
\usepackage[numbers,sort&compress]{natbib}

\usepackage[colorlinks=true, pdfstartview=FitV, linkcolor=blue, citecolor=blue, urlcolor=blue]{hyperref}

\usepackage{yfonts}

\usepackage{marvosym}

\newtheorem{theorem}			     {Theorem} [section]
\newtheorem{proposition}[theorem]	 {Proposition}	
\newtheorem{corollary}	  [theorem]	 {Corollary}	
\newtheorem{lemma}	      [theorem]  {Lemma}		
\theoremstyle{definition}

\newtheorem{remark}{Remark}

\newcommand{\C}{\mathbb{C}}
\newcommand{\R}{\mathbb{R}}

\newcommand{\Or}{\mathcal{O}}

\newcommand{\Be}{{\mathrm{Be}}}

\def\gg{{\mathfrak g}}

\newcommand{\res}{\textrm{Res}\,}

\newcommand{\re}{\textrm{Re}\,}
\newcommand{\im}{\textrm{Im}\,}

\usepackage{tikz}
\usetikzlibrary{arrows}
\usetikzlibrary{decorations.pathmorphing}
\usetikzlibrary{decorations.markings}
\usetikzlibrary{patterns}
\usetikzlibrary{automata}
\usetikzlibrary{positioning}
\usepackage{tikz-cd}

\usepackage{pgfplots}

\numberwithin{equation}{section}

\def\ds{\displaystyle}
\def\bigO{{\cal O}}

\tikzset{->-/.style={decoration={
				markings,
				mark=at position #1 with {\arrow[thick]{>}}},postaction={decorate}}}
	
	\tikzset{-<-/.style={decoration={
				markings,
				mark=at position #1 with {\arrow[thick]{<}}},postaction={decorate}}}
				
\begin{document}
\title{The Bessel kernel determinant on large intervals \\ and Birkhoff's ergodic theorem}
\author{Elliot Blackstone, Christophe Charlier, Jonatan Lenells \footnote{Department of Mathematics, KTH Royal Institute of Technology, 100 44 Stockholm, Sweden. The work of all three authors is supported by the European Research Council, Grant Agreement No. 682537. Lenells also acknowledges support from the Swedish Research Council, Grant No. 2015-05430, and the Ruth and Nils-Erik Stenb\"ack Foundation.
E-mail: elliotb@kth.se, cchar@kth.se, jlenells@kth.se.}}
\maketitle

\begin{abstract}
The Bessel process models the local eigenvalue statistics near $0$ of certain large positive definite matrices. In this work, we consider the probability
\begin{align*}
\mathbb{P}\Big( \mbox{there are no points in the Bessel process on } (0,x_{1})\cup(x_{2},x_{3})\cup\cdots\cup(x_{2g},x_{2g+1}) \Big),
\end{align*}
where $0<x_{1}<\cdots<x_{2g+1}$ and $g \geq 0$ is any non-negative integer. We obtain asymptotics for this probability as the size of the intervals becomes large, up to and including the oscillations of order $1$. In these asymptotics, the most intricate term is a one-dimensional integral along a linear flow on a $g$-dimensional torus, whose integrand involves ratios of Riemann $\theta$-functions associated to a genus $g$ Riemann surface.  We simplify this integral in two generic cases: (a) If the flow is ergodic, we compute the leading term in the asymptotics of this integral explicitly using Birkhoff's ergodic theorem. (b) If the linear flow has certain ``good Diophantine properties'', we obtain improved estimates on the error term in the asymptotics of this integral.  In the case when the flow is both ergodic and has ``good Diophantine properties'' (which is always the case for $g=1$, and ``almost always'' the case for $g \geq 2$), these results can be combined, yielding particularly precise and explicit large gap asymptotics.
\end{abstract}




\section{Introduction}
Let $g \geq 0$ be an integer, let $\vec{x} = (x_1,x_2,\dots,x_{2g+1})$ be such that $0<x_1<\cdots<x_{2g+1}<+\infty$, and consider the Fredholm determinant
\begin{equation}\label{generating}
F(\vec{x}) = \det\left( 1 - \mathcal{K}|_{\mathcal{I}_g} \right), \qquad \mathcal{I}_g=(0,x_{1})\cup(x_{2},x_{3})\cup\cdots\cup(x_{2g},x_{2g+1}),
\end{equation}
where $\mathcal{K}|_{\mathcal{I}_g}$ is the trace class operator acting on $L^{2}(\mathcal{I}_g)$ whose kernel is given by
\begin{equation}\label{Bessel kernel}
K(x,y) = \frac{J_{\alpha}(\sqrt{x})\sqrt{y}J_{\alpha}^{\prime}(\sqrt{y})-\sqrt{x}J_{\alpha}^{\prime}(\sqrt{x})J_{\alpha}(\sqrt{y})}{2(x-y)}, \qquad \alpha > -1,
\end{equation}
and $J_\alpha$ is the Bessel function of the first kind of order $\alpha$. 

\medskip The determinantal point process on $(0,+\infty)$ associated to the kernel \eqref{Bessel kernel} is referred to as the Bessel point process. This is a universal point process in random matrix theory which models the behavior of the smallest eigenvalues for a wide class of large positive definite random matrices, see e.g. \cite{For1, ForNag}. As is well-known, $F(\vec{x})$ is the probability of finding no points on $\mathcal{I}_{g}$ in the Bessel process. In this paper, we obtain asymptotics for $F(r\vec{x}) = F(rx_1,\dots,rx_{2g+1})$ as $r \to + \infty$. 

\paragraph{Known results for $g=0$.} Tracy and Widom in \cite{TraWidLUE} have shown that $F(x_{1})$ can be naturally expressed in terms of the solution of a Painlev\'{e} V equation. They also obtained large $r$ asymptotics of $F(rx_{1})$, namely,
\begin{equation}\label{largegapBessel1}
F(rx_{1}) = \det\left( 1 - \mathcal{K}|_{\mathcal{I}_0} \right) = \exp\left(-\frac{rx_{1}}{4}+\alpha\sqrt{rx_{1}}-\frac{\alpha^2}{4}\log r +C+\bigO(r^{-1/2})\right) \qquad \mbox{as } r \to + \infty,
\end{equation}
where $C$ is independent of $r$. Based on numerics, they conjectured that
\begin{align}\label{Cgenus0}
C=G(1+\alpha)(2\pi)^{-\frac{\alpha}{2}}-\frac{\alpha^{2}}{4}\log x_{1},
\end{align}
where $G$ is Barnes' $G$-function. The expression (\ref{Cgenus0}) for the constant $C$ was first established rigorously by Ehrhardt in \cite{Ehr2010} for $\alpha \in (-1,1)$, and then by Deift, Krasovsky, and Vasilevska in \cite{DeiftKrasVasi} for all values of $\alpha \in (-1,+\infty)$. 

\medskip For $g\geq 1$, it is known from \cite{ChDoe} that $F(\vec{x})$ is related to a solution of a system of $2g+1$ coupled Painlev\'{e} V equations. However, to the best of our knowledge, there exist no results prior to this work on the large $r$ asymptotics of $F(r\vec{x})$ for $g \geq 1$.  

\paragraph{Statement of results.}  Let $X$ be the Riemann surface of genus $g$ associated to $\sqrt{\mathcal{R}(z)}$, where
\begin{align}\label{intro: R}
\mathcal{R}(z)=\prod_{j=1}^{2g+1}(z+x_j).
\end{align}
We view $X$ as two copies of $\mathbb{C}$ that are glued along $(-\infty,-x_{2g+1})\cup \cdots \cup (-x_{4},-x_{3}) \cup (-x_{2},-x_{1})$. Given $z \in \mathbb{C}$, we write $z^{+} \in X$ (resp. $z^{-} \in X$) for the points on the first (resp. second) sheet of $X$ that project onto $z$. We choose the sheets such that $\sqrt{\mathcal{R}(z^{+})}>0$ for $z\in(-x_1,+\infty)$. Consider the cycles $A_{1},\ldots,A_{g},B_{1},\ldots,B_{g}$ shown in Figure \ref{fig:homology}. The cycle $B_{j}$ lies entirely in the first sheet and surrounds $(-x_{2j},-x_{2j-1})$ in the positive direction, while $A_{j}$ lies partly in the upper half-plane of the first sheet (the solid red curves in Figure \ref{fig:homology}) and partly in the lower half-plane of the second sheet (the dashed red curves in Figure \ref{fig:homology}). These cycles form a canonical homology basis of $X$. Define the $g\times g$ matrix $\mathbb{A}$ and the column vector $\vec{a}$ by
\begin{align*}
\mathbb{A} = (a_{i,j})_{i,j=1}^{g}, \qquad \vec{a} =(a_{1,g+1} \; \; \cdots \;\; a_{g,g+1})^{t},
\end{align*}
where
\begin{align*}
a_{i,j}=\oint_{A_i}\frac{s^{j-1}}{\sqrt{\mathcal{R}(s)}}ds, \qquad i=1,\ldots,g, \quad j=1,\ldots,g+1,
\end{align*}
and $^t$ denotes the transpose operation. Because $x_{1},\ldots,x_{2g+1}$ are real, it follows that $a_{i,j}\in\mathbb{R}$ for all $i=1,\ldots,g$, $j=1,\ldots,g+1$. Let $\vec{\omega}=(\omega_1,\dots,\omega_g)$ be the row vector of holomorphic one-forms given by
\begin{align}\label{intro: omegaVec}
\vec{\omega}=\frac{dz}{\sqrt{\mathcal{R}(z)}} (1 \;\; z \;\; \cdots \; \; z^{g-1})\mathbb{A}^{-1},
\end{align}
and let $q$ be the following polynomial of degree $g$:
\begin{align}\label{intro: q}
q(z)=\frac{z^g}{2}+\sum_{j=0}^{g-1}q_jz^j, \quad \text{where} \quad  (q_{0},\ldots,q_{g-1})^{t} = -\frac{1}{2}\mathbb{A}^{-1} \vec{a}.
\end{align}
\begin{figure}
\vspace{-1cm}
\centering
\begin{tikzpicture}

\draw (-8,0) -- (-6,0);
\draw (-5,0) -- (-3.6,0);
\draw (-2,0) -- (0,0);
\draw[dotted] (0.5,0) -- (1,0);
\draw (1.5,0) -- (2.7,0);
\draw (3.9,0) -- (6.5,0);

\draw[fill] (-6,0) circle (0.05cm);
\draw[fill] (-5,0) circle (0.05cm);
\draw[fill] (-3.6,0) circle (0.05cm);
\draw[fill] (-2,0) circle (0.05cm);
\draw[fill] (0,0) circle (0.05cm);
\draw[fill] (1.5,0) circle (0.05cm);
\draw[fill] (2.7,0) circle (0.05cm);
\draw[fill] (3.9,0) circle (0.05cm);
\draw[fill] (6.5,0) circle (0.05cm);

\node at (-6.0,-0.3) {$-x_{2g+1}$};
\node at (-4.82,-0.3) {$-x_{2g}$};
\node at (-3.6,-0.3) {$-x_{2g-1}$};
\node at (-1.75,-0.3) {$-x_{2g-2}$};
\node at (-0.30,-0.3) {$-x_{2g-3}$};
\node at (1.60,-0.3) {$-x_{4}$};
\node at (2.6,-0.3) {$-x_{3}$};
\node at (3.9,-0.3) {$-x_{2}$};
\node at (6.35,-0.3) {$-x_{1}$};

\draw[blue] (-4.1,-0.10) ellipse (1.2cm and 0.70cm);
\draw[blue] (-1.0,-0.10) ellipse (1.40cm and 0.80cm);
\draw[blue] (2.1,-0.10) ellipse (1.0cm and 0.70cm);
\draw[blue] (5.2,-0.10) ellipse (1.70cm and 0.80cm);

\node[blue] at (-3.15,-0.92) {$B_g$};
\node[blue] at (0.4,-0.92) {$B_{g-1}$};
\node[blue] at (2.9,-0.92) {$B_2$};
\node[blue] at (6.65,-0.92) {$B_1$};

\draw[blue,arrows={-Triangle[length=0.18cm,width=0.12cm]}]
(-3.95,-0.8) --  ++(0:0.001);
\draw[blue,arrows={-Triangle[length=0.18cm,width=0.12cm]}]
(-0.85,-0.9) --  ++(0:0.001);
\draw[blue,arrows={-Triangle[length=0.18cm,width=0.12cm]}]
(2.25,-0.8) --  ++(0:0.001);
\draw[blue,arrows={-Triangle[length=0.18cm,width=0.12cm]}]
(5.4,-0.9) --  ++(0:0.001);

\draw[red] (-4.2,0) arc(0:180:1.25cm and 0.85cm);
\draw[red,dashed] (-6.7,0) arc(180:360:1.25cm and 0.85cm);
\draw[red] (-7.0,0) .. controls (-7.0,1.7) and (-1.00,1.7) .. (-1.00,0);
\draw[red,dashed] (-7.0,0) .. controls (-7.0,-1.7) and (-1.00,-1.7) .. (-1.00,0);
\draw[red] (-7.3,0) .. controls (-7.3,2.3) and (2.15,2.3) .. (2.15,0);
\draw[red,dashed] (-7.3,0) .. controls (-7.3,-2.3) and (2.15,-2.3) .. (2.15,0);
\draw[red] (-7.6,0) .. controls (-7.6,2.9) and (5.00,2.9) .. (5.00,0);
\draw[red,dashed] (-7.6,0) .. controls (-7.6,-2.9) and (5.00,-2.9) .. (5.00,0);

\node[red] at (-4.85,0.92) {$A_g$};
\node[red] at (-2.30,1.32) {$A_{g-1}$};
\node[red] at (1.40,1.32) {$A_2$};
\node[red] at (4.55,1.32) {$A_1$};

\draw[red,arrows={-Triangle[length=0.18cm,width=0.12cm]}]
(-5.55,0.85) --  ++(0:-0.001);
\draw[red,arrows={-Triangle[length=0.18cm,width=0.12cm]}]
(-5.35,-0.85) --  ++(0:0.001);
\draw[red,arrows={-Triangle[length=0.18cm,width=0.12cm]}]
(-3.85,1.27) --  ++(0:-0.001);
\draw[red,arrows={-Triangle[length=0.18cm,width=0.12cm]}]
(-3.65,-1.27) --  ++(0:0.001);
\draw[red,arrows={-Triangle[length=0.18cm,width=0.12cm]}]
(-2.2,1.72) --  ++(0:-0.001);
\draw[red,arrows={-Triangle[length=0.18cm,width=0.12cm]}]
(-2.0,-1.72) --  ++(0:0.001);
\draw[red,arrows={-Triangle[length=0.18cm,width=0.12cm]}]
(-1.5,2.18) --  ++(0:-0.001);
\draw[red,arrows={-Triangle[length=0.18cm,width=0.12cm]}]
(-1.3,-2.18) --  ++(0:0.001);

\end{tikzpicture}
\vspace{-0.5cm}
\caption{Canonical homology basis.}
\label{fig:homology}
\end{figure}
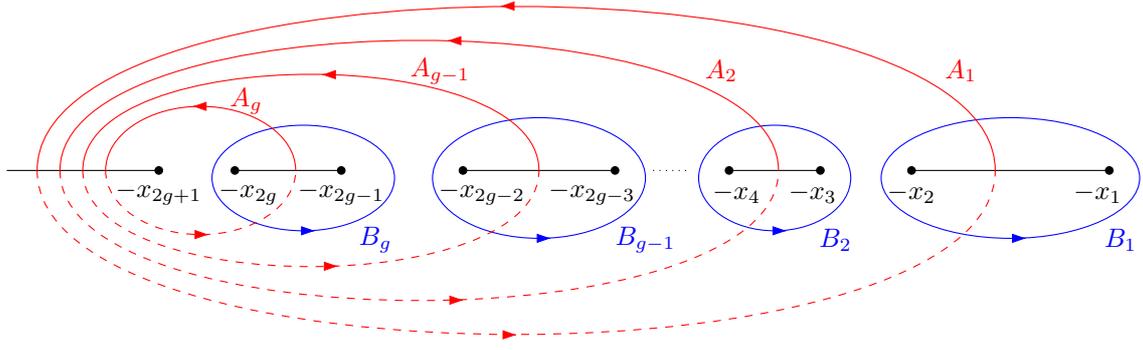
By definition, $\vec\omega$ and $q$ satisfy
\begin{align}\label{prop of omega and q in intro}
\oint_{A_k}\omega_j=\delta_{jk}, \qquad j,k = 1, \ldots,g, \qquad \mbox{ and } \qquad \oint_{A_{j}}\frac{q(s)}{\sqrt{R(s)}}ds=0, \qquad j=1,\dots,g.
\end{align}
Since $a_{i,j}\in\mathbb{R}$, $q$ has real coefficients, and we show in Section \ref{subsection: g-function} that $q$ has exactly one zero in each of the intervals $(-x_{2j+1},-x_{2j})$, $j=1,\ldots,g$.  The constant $c$ defined by
\begin{align}\label{intro: c const}
c=q_{g-1}-\frac{1}{4}\sum_{j=1}^{2g+1}x_j,
\end{align}
and the constants $\Omega_j$, $j=1,\dots,g$, defined by
\begin{align}\label{intro: Omegaj}
\Omega_j= -i \oint_{B_{j}}\frac{q(s)ds}{\sqrt{\mathcal{R}(s)}} = -2\int_{-x_{2j-1}}^{-x_{2j}}\frac{q(s)ds}{|\mathcal{R}(s)|^\frac{1}{2}} \in (0,+\infty)
\end{align}
appear in our final results. Define
\begin{align}\label{def of d1}
d_1= \frac{1}{2}\sum_{j=0}^g\tilde{\alpha}_j a_{j,g+1}, \qquad a_{0,j}=-2 \, \re \int_{-x_{2g+1}}^{0^{+}}\frac{s^{j-1}ds}{\sqrt{\mathcal{R}(s)}}, \; j=1,\ldots,g+1,
\end{align}
where the path of integration lies in the upper half-plane of the first sheet and $\tilde{\alpha}_{0},\tilde{\alpha}_{1},\ldots,\tilde{\alpha}_{g} \in \mathbb{R}$ are given by
\begin{align}\label{def of alpha tilde}
\tilde{\alpha}_0 =\frac{\alpha}{2}, \qquad (\tilde{\alpha}_{1} \;\; \tilde{\alpha}_{2} \;\; \cdots \;\; \tilde{\alpha}_{g}) = -\frac{\alpha}{2} (a_{0,1} \;\; a_{0,2} \;\; \ldots \;\; a_{0,g}) \mathbb{A}^{-1}.
\end{align}
Given $r \geq 0$, we let $\vec\nu(r) = \vec\nu = (\nu_{1},\ldots, \nu_{g})^{t}$ be the column vector with components
\begin{align}\label{intro: nu}
\nu_j=-\tilde{\alpha}_j -\Omega_j\frac{\sqrt{r}}{2\pi} \in \mathbb{R}, \qquad j=1,\ldots,g.
\end{align}
It follows from the general theory of Riemann surfaces, see e.g. \cite[p. 63]{FK1992}, that the period matrix
\begin{equation}\label{intro: tau}
\tau:=\left(\oint_{B_i}\omega_j\right)_{i,j=1,\dots,g}
\end{equation}
is symmetric and has positive definite imaginary part. The associated Riemann $\theta$-function is defined by
\begin{equation}\label{intro: theta def}
\theta:\mathbb{C}^g\to\mathbb{C}, \qquad \vec{z} \mapsto \theta(\vec{z})=\sum_{\vec{n}\in\mathbb{Z}^g}e^{i\pi(\vec{n}^t\tau\vec{n}+2\vec{n}^t\vec{z})}.
\end{equation}
This function is entire and satisfies 
\begin{equation}\label{intro: ThetaPeriodic}
\theta(-\vec{z}) = \theta(\vec{z}), \qquad \theta(\vec{z}+\vec{\lambda}'+\tau\vec{\lambda})=e^{-i\pi(2\vec{\lambda}^t\vec{z}+\vec{\lambda}^t\tau\vec{\lambda})}\theta(\vec{z}), \qquad \mbox{for all }\vec{z} \in \mathbb{C}^{g}, \quad \vec{\lambda}, \vec{\lambda}'\in\mathbb{Z}^g.
\end{equation}
The Abel map $\vec{\varphi}$ is defined by
\begin{align}\label{def of Abel map}
\vec{\varphi}: X \to J(X), \qquad P \mapsto \vec{\varphi}(P)=\int_{-x_1}^P \vec{\omega}^{t},
\end{align}
where $J(X)$ is the Jacobian variety of $X$: 
\begin{align}\label{def of J(X) and L(X)}
J(X)=\mathbb{C}^g/L(X), \quad \mbox{ where } \quad L(X)=\{\vec{z}\in\mathbb{C}^g:\vec{z}=\vec{n}+\tau\vec{m} ~ \text{for some} ~ \vec{n},\vec{m}\in\mathbb{Z}^g\}.
\end{align}
We show in Proposition \ref{ThetaIdentity} that for each $r\geq 0$, the multi-valued function $P \mapsto \theta(\vec{\varphi}(P)+\frac{\vec{e}_1}{2}+ \vec\nu(r))$ has a well-defined zero set of cardinality $g$, and that the projections of these zeros on $\mathbb{C}$, which are denoted $b_{1}(\vec{\nu}(r)),\ldots,b_{k}(\vec{\nu}(r))$, satisfy $b_{k}(\vec{\nu}(r)) \in [-x_{2k+1},-x_{2k}]$, $k=1,\ldots,g$. Define the function $\mathcal{B} :\mathbb{C}\times \mathbb{R}^{g}/\mathbb{Z}^{g} \to\mathbb{C}$ by
\begin{align}\label{calBdef}
\mathcal{B}(z,\vec{u})&=\frac{\prod_{k=1}^g(z-b_k(\vec{u}))}{q(z)}, \qquad z \in \mathbb{C}, \; \vec{u} \in \mathbb{R}^{g}/\mathbb{Z}^{g}.
\end{align}
We mention that $\mathcal{B}$ admits another equivalent (but more complicated) expression involving certain ratios of $\theta$-functions, see \eqref{B}.

Our first theorem provides an asymptotic formula for $F(r\vec{x})$ for any choice of the points $0<x_1<x_2<\cdots<x_{2g+1}<+\infty$.

\begin{theorem}[The general case]\label{thm: main result - DIZ analogue}
Let $0<x_1<x_2<\cdots<x_{2g+1}<+\infty$ and $\alpha>-1$ be fixed. As $r \to + \infty$, we have
\begin{align}
& F(r\vec{x}) = \exp \Bigg( c\, r - d_1 \sqrt{r} + \frac{1-4\alpha^{2}}{16} \log r + \log \theta(\vec\nu(r)) - \frac{1}{32} \sum_{j=1}^{2g+1}\int_{M}^{r} \mathcal{B}(-x_j,\vec{\nu}(t))\frac{dt}{t}+\widetilde{C}+ \bigO(r^{-\frac{1}{2}}) \Bigg), \label{asymptotics in main thm}
\end{align}
where $F$ is the Fredholm determinant \eqref{generating}, $M>0$ is arbitrary but independent of $r$ and $\widetilde{C}=\widetilde{C}(M)$ is independent of $r$. Furthermore, for any $j=1,\ldots,2g+1$, as $r \to + \infty$ we have
\begin{align}\label{leading term is a space average}
& \int_{M}^{r} \mathcal{B}(-x_j,\vec{\nu}(t))\frac{dt}{t} = \mathcal{B}_{j} \log r + o(\log r), \qquad \mbox{ where } \qquad \mathcal{B}_{j} := \lim_{\mathrm{T} \to + \infty}\frac{1}{\mathrm{T}} \int_{0}^{\mathrm{T}} \mathcal{B}(-x_j,\vec{\nu}(t^{2}))dt.
\end{align}
\end{theorem}

\begin{remark}
If $g=0$, then the term $\log \theta(\vec{\nu}(r))$ in (\ref{asymptotics in main thm}) should be interpreted as being equal to $0$.
\end{remark}

The asymptotic formula \eqref{leading term is a space average} is the best we can say for general choices of $x_{1},\ldots,x_{2g+1}$. However, if the linear flow
\begin{align}\label{the linear flow intro}
(0,+\infty) \ni r \mapsto (\nu_{1}(r^{2}) \hspace{-2.4mm} \mod 1 \; , \, \nu_{2}(r^{2}) \hspace{-2.4mm} \mod 1 \; , \ldots, \, \nu_{g}(r^{2})\hspace{-2.4mm} \mod 1)
\end{align}
has either ``good Diophantine properties" or is ergodic in the $g$-dimensional torus $\mathbb{R}^{g}/ \mathbb{Z}^{g}$, then one can strengthen \eqref{leading term is a space average} (and thereby \eqref{asymptotics in main thm}). Let us define these notions:
\begin{itemize}
\item The flow \eqref{the linear flow intro} is said to have ``good Diophantine properties"
if there exist $\delta_{1},\delta_{2}>0$ such that 
\begin{align}\label{good diophantine prop}
|\vec{m}^{t}\vec{\Omega}| \geq \delta_{1}\| \vec{m} \|^{-\delta_{2}} \qquad \mbox{for all } \vec{m} \in \mathbb{Z}^{g \times 1} \mbox{ such that } \vec{m}^{t}\vec{\Omega} \neq 0,
\end{align}
where $\vec{\Omega} = (\Omega_{1} \;  \; \Omega_{2} \;  \; \ldots \;  \; \Omega_{g})^{t}$ and $\| \vec{m} \|^{2} = \vec{m}^{t}\vec{m}$.

\item The flow \eqref{the linear flow intro} is said to be ergodic in $\mathbb{R}^{g}/ \mathbb{Z}^{g}$ if the set $\{\vec{\nu}(r^{2}) \hspace{-0.1mm} \mod \mathbb{Z}^{g}\}_{r \in (0,+\infty)}$ is dense in $\mathbb{R}^{g}/ \mathbb{Z}^{g}$.  Equivalently, the flow \eqref{the linear flow intro} is ergodic in $\mathbb{R}^{g}/ \mathbb{Z}^{g}$ if $\Omega_1,\ldots,\Omega_g$ are rationally independent, that is, if $\Omega_1,\ldots,\Omega_g$ satisfy the following condition: If $(n_{1},n_{2},\ldots,n_{g}) \in \mathbb{Z}^{g}$ is such that
\begin{align}\label{ergodic condition in thm}
\Omega_{1} n_{1}+\Omega_{2} n_{2} + \cdots + \Omega_{g} n_{g} = 0,
\end{align}
then $n_{1}=n_{2}= \cdots=n_{g}=0$.
\end{itemize}
Let $\mathcal{D},\mathcal{E} \subseteq (0,+\infty)^{g}$ be  given by
\begin{align*}
\mathcal{D}=\{(\Omega_{1},\ldots,\Omega_{g}): \mbox{\eqref{good diophantine prop} holds true}\}, \qquad \mathcal{E}=\{(\Omega_{1},\ldots,\Omega_{g}): \mbox{\eqref{ergodic condition in thm} holds true}\}.
\end{align*}
The sets $\mathcal{D}$ and $\mathcal{E}$ are generic, in the sense that $(0,+\infty)^{g}\setminus \mathcal{D}$ and $(0,+\infty)^{g}\setminus \mathcal{E}$ have Lebesgue measure zero. In the case of $\mathcal{E}$, this follows easily from the definition; in the case of $\mathcal{D}$, this follows from the Khintchine--Groshev theorem, see e.g. \cite[Corollary 1.4 (b)]{BD1999}. If $g=1$, then $\mathcal{D}=\mathcal{E}=(0,+\infty)$, but for $g \geq 2$, it is easy to see that $\mathcal{D}\subsetneq \mathcal{E}$, $\mathcal{E}\subsetneq \mathcal{D}$, and $\mathcal{D}\neq (0,+\infty)^{g} \neq\mathcal{E}$. For example, 
\begin{itemize}
\item \vspace{-0.2cm} if $g=2$ then $(1,\sqrt{2}) \in \mathcal{D}\cap \mathcal{E}$,\footnote{$(1,\sqrt{2}) \in \mathcal{D}$ holds with $\delta_{2}=1$ in \eqref{good diophantine prop}. This follows from Liouville's theorem on diophantine approximation and the fact that $\sqrt{2}$ is a root of the degree $2$ polynomial $x^{2}-2$.}
\item \vspace{-0.15cm} if $g=2$ then $(1,0) \in \mathcal{D}\setminus \mathcal{E}$, 
\item \vspace{-0.15cm} if $g=2$ then $\big(1,\mathcal{L}\big) \in \mathcal{E}\setminus \mathcal{D}$,  where $\mathcal{L} = \sum_{n=1}^{+\infty}10^{-n!}$ is Liouville's constant,
\item \vspace{-0.2cm} if $g=3$ then $\big(1,\mathcal{L},0\big)\notin \mathcal{D}\cup \mathcal{E}$.
\end{itemize}
We show in Appendix \ref{appendix: mapping open ball} that the image of the mapping $(x_{1},x_{2},\ldots,x_{2g+1})\mapsto (\Omega_{1},\Omega_{2},\ldots,\Omega_{g})$ contains an open ball in $(0,+\infty)^{g}$, so that each of the four different cases
\begin{align}\label{the four cases}
(\Omega_{1},\ldots,\Omega_{g})\notin \mathcal{D}\cup \mathcal{E}, \quad (\Omega_{1},\ldots,\Omega_{g}) \in \mathcal{D}\setminus \mathcal{E}, \quad (\Omega_{1},\ldots,\Omega_{g})\in \mathcal{E}\setminus \mathcal{D}, \quad (\Omega_{1},\ldots,\Omega_{g}) \in \mathcal{D}\cap \mathcal{E}
\end{align}
can and do occur for certain choices of $(x_{1},x_{2},\ldots,x_{2g+1})$. (In fact, for $g=2$ there are only three cases because $\mathcal{D}\cup \mathcal{E} = (0,+\infty)^{2}$).

\begin{theorem}[The ``good Diophantine" case]\label{thm: main result - periodic}
Let $0<x_1<x_2<\cdots<x_{2g+1}<+\infty$ and $\alpha>-1$ be fixed, and assume that $(\Omega_{1},\Omega_{2},\ldots,\Omega_{g}) \in \mathcal{D}$. In this case, \eqref{leading term is a space average} can be improved to
\begin{align}\label{good error term inside main thm dioph}
& \int_{M}^{r} \mathcal{B}(-x_j,\vec{\nu}(t))\frac{dt}{t} = \mathcal{B}_{j} \log r + C_{j} + \bigO(r^{-1/2}) \qquad \mbox{as } r \to + \infty, \quad j=1,\ldots,2g+1,
\end{align}
where $C_{1},\ldots,C_{2g+1}$ are independent of $r$. Therefore, \eqref{asymptotics in main thm} can be improved to
\begin{align}
& F(r\vec{x}) = \exp \Bigg( c\, r - d_1 \sqrt{r} + \Bigg(\frac{1-4\alpha^2}{16}-\frac{1}{32}\sum_{j=1}^{2g+1}\mathcal{B}_{j} \Bigg) \log r +\log\theta(\vec\nu(r))+C+\bigO(r^{-1/2}) \Bigg), \nonumber
\end{align}
as $r \to + \infty$, where $C$ is independent of $r$. 
\end{theorem}
\begin{remark}
The constant $C$ in Theorem \ref{thm: main result - periodic} equals $\widetilde{C}-\frac{1}{32}(C_{1}+\cdots+C_{2g+1})$, where $\widetilde{C}$ is as in \eqref{asymptotics in main thm} and $C_{1},\ldots,C_{2g+1}$ are as in Theorem \ref{thm: main result - periodic}.
\end{remark}

\begin{theorem}[The ergodic case]\label{thm: main result - ergodic}
Let $0<x_1<x_2<\cdots<x_{2g+1}<+\infty$ and $\alpha>-1$ be fixed, and assume that $(\Omega_{1},\Omega_{2},\ldots,\Omega_{g}) \in \mathcal{E}$. 
In this case, \eqref{leading term is a space average} can be improved to
\begin{align}\label{explicit leading term}
& \int_{M}^{r} \mathcal{B}(-x_j,\vec{\nu}(t))\frac{dt}{t} = 2 \log r + o(\log r) \qquad \mbox{as } r \to + \infty, \quad j=1,\ldots,2g+1.
\end{align}
Therefore, \eqref{asymptotics in main thm} can be improved to
\begin{align*}
& F(r\vec{x}) = \exp \left( c\, r - d_1 \sqrt{r} - \frac{g+2\alpha^2}{8} \log r+o(\log r) \right) \qquad \mbox{as } r \to + \infty.
\end{align*}
\end{theorem}

By combining Theorems \ref{thm: main result - periodic} and \ref{thm: main result - ergodic}, we immediately obtain the following.

\begin{corollary}[The ``good Diophantine" and ergodic case]\label{cor: ergo+dio}
Let $0<x_1<x_2<\cdots<x_{2g+1}<+\infty$ and $\alpha>-1$ be fixed, and assume that $(\Omega_{1},\Omega_{2},\ldots,\Omega_{g}) \in \mathcal{D} \cap \mathcal{E}$. Then
\begin{align*}
& F(r\vec{x}) = \exp \left( c\, r - d_1 \sqrt{r} - \frac{g+2\alpha^2}{8} \log r+\log\theta(\vec\nu(r))+ C + \bigO(r^{-\frac{1}{2}}) \right) \qquad \mbox{as } r \to + \infty,
\end{align*}
where $C$ is independent of $r$.
\end{corollary}
\begin{remark}
Corollary \ref{cor: ergo+dio} has been verified numerically for $g=1,2,3$ and several choices of $x_{1},\ldots,x_{2g+1}$ using \cite{Bornemann}.
\end{remark}

The case $g=1$ is simpler because both \eqref{good diophantine prop} and \eqref{ergodic condition in thm} are automatically satisfied for any choice of $0<x_1<x_2<x_3<+\infty$. Hence, Corollary \ref{cor: ergo+dio} always applies when $g = 1$, and we get the following result.

\begin{corollary}[The case $g=1$]\label{coro:g=1}
Let $g=1$ and fix $\alpha>-1$, $0<x_1<x_2<x_3<+\infty$. As $r\to+\infty$, 
\begin{align*}
F(r(x_{1},x_{2},x_{3}))=\exp\left(c\,r-d_1\sqrt{r}-\frac{1+2\alpha^2}{8}\log r +\log\theta(\nu_{1}(r))+C+\bigO(r^{-1/2})\right),
\end{align*}
where $C$ is independent of $r$.
\end{corollary}

\begin{remark}\label{remark:g=0}
(The case $g=0$). If $g=0$, then \eqref{intro: c const} gives $c=-\frac{x_{1}}{4}$ and \eqref{def of d1} gives
\begin{align*}
d_{1} = \frac{\alpha}{4}a_{0,1} = -\frac{\alpha}{2} \int_{-x_1}^0\frac{ds}{\sqrt{s+x_1}} = - \alpha \sqrt{x_{1}}.
\end{align*}
Since both \eqref{good diophantine prop} and \eqref{ergodic condition in thm} are automatically satisfied for $g=0$, using Corollary \ref{cor: ergo+dio}, we obtain
\begin{align*}
F(rx_{1})=\exp\left(-\frac{rx_{1}}{4}+\alpha\sqrt{rx_{1}}-\frac{\alpha^2}{4}\log r +C+\bigO(r^{-1/2})\right) \qquad \mbox{as } r \to + \infty,
\end{align*}
which is consistent with \eqref{largegapBessel1}.
\end{remark}

\paragraph{Related work.} The determination of large gap asymptotics is a classical problem in random matrix theory with a long history. There exist various results on large gap asymptotics in the case of a gap on a \textit{single} interval (the so-called ``one-cut regime"), see \cite{dCM1973, D1976, W1971, K2003, EhrSine, DIKZ2007} for the sine process, \cite{DIK, BBD2008} for the Airy process, \cite{Ehr2010, DeiftKrasVasi} for the Bessel process, \cite{CGS2019, CLM1, CLM2} for the Wright's generalized Bessel and Meijer-$G$ point processes, \cite{DXZ2020} for the Pearcey process, \cite{BW1983, BB1995, BDIK2015, CharlierSine, BothnerBuckingham, ChCl3, CaCl, CCR2020, BIP, CharlierBessel, ChCl4, DXZ2020 thinning} for thinned-deformations of these universal point processes, and \cite{VV2010, RRZ2011, DVAiry2013} for the sine-$\beta$, Airy-$\beta$ and Bessel-$\beta$ point processes. We also refer to \cite{K2009} and \cite{F2014} for two overviews.

Large gap asymptotics in the case of {\it multiple} intervals has been much less explored. For the sine process, this problem was first investigated by Widom \cite{Widom1995}, who obtained an explicit expression for the leading term in the asymptotics, and characterized the oscillations in terms of the solution to a Jacobi inversion problem. These oscillations were later described more explicitly by Deift, Its, and Zhou \cite{DIZ} in terms of $\theta$-functions. Actually, the results of \cite{DIZ} are, in many respects, similar to the results of our first two theorems (Theorem \ref{thm: main result - DIZ analogue} and Theorem \ref{thm: main result - periodic}). In particular, the asymptotic formula of \cite{DIZ} for the probability to observe $g$ gaps in the sine process involves (1) an explicit leading term, (2) some oscillations of order $1$ described in terms of $\theta$-functions, and (3) a rather complicated integral involving ratios of $\theta$-functions. As in our case (see (\ref{leading term is a space average})), the latter integral grows logarithmically with the size of the intervals. It was also noted in \cite{DIZ} that, in the generic case where a certain vector has ``good Diophantine properties", this integral can be estimated with good control of the error term (this idea of \cite{DIZ} has been the main inspiration for our Theorem \ref{thm: main result - periodic}). However, the coefficient of the logarithmic term was written in \cite{DIZ}  as a time average of ratios of $\theta$-functions (cf. our formula \eqref{leading term is a space average}) and was not simplified. Recently, for the case of two intervals (this corresponds to a genus $1$ situation), Fahs and Krasovsky in \cite{FahsKrasovsky2} showed that this coefficient is identically equal to $-1/2$. Large gap asymptotics for the Airy process in certain genus $1$ situations were recently computed in \cite{BCL2020, BCL2020 2, KraMarou}, and there too simple expressions for the log coefficients appearing in the large gap asymptotics were obtained. We also mention that the multiplicative constants in the asymptotics were explicitly computed in \cite{FahsKrasovsky2, KraMarou}.

In Theorem \ref{thm: main result - ergodic}, we provide a simplified formula for the coefficient of the logarithmic term appearing in the large $r$ asymptotics of $F(r\vec{x})$ for an arbitrary $g \geq 0$, in the generic case when the flow \eqref{the linear flow intro} is ergodic in $\mathbb{R}^{g}/ \mathbb{Z}^{g}$. The proof of Theorem \ref{thm: main result - ergodic} requires the novel use of Birkhoff's ergodic theorem. More precisely, using the ergodicity of the flow \eqref{the linear flow intro}, we rewrite the one-dimensional integrals $\mathcal{B}_{j}$ of \eqref{leading term is a space average} (which we call \textit{time averages}) as $g$-dimensional integrals over $\mathbb{R}^{g}/ \mathbb{Z}^{g}$ (which we call \textit{space averages}). Quite remarkably, these space averages can in turn be evaluated explicitly after a non-trivial change of variables involving the Abel map, see Proposition \ref{varphiAprop} and Lemma \ref{orientationlemma}. 
We expect that a similar simplification of the logarithmic term can be achieved by means of Birkhoff's ergodic theorem for any genus $g \geq 1$ also for other point processes such as the sine and Airy processes.
As noted above Corollary \ref{coro:g=1}, the case $g=1$ is somewhat simpler in several respects; in particular, for $g=1$ \textit{every} flow is both ergodic and has ``good Diophantine properties". Hence there is no need to invoke Birkhoff's ergodic theorem when $g = 1$, see \cite{FahsKrasovsky2, BCL2020, BCL2020 2, KraMarou}. For these reasons, we feel that Theorem \ref{thm: main result - ergodic} is an important novel contribution of the current paper.

\paragraph{Overview of the paper.}  In Section \ref{section:diffid}, we link the Fredholm determinant $F(r\vec{x})$ to the solution $\Phi$ of a certain Riemann--Hilbert (RH) problem. In Sections \ref{section:RH1}--\ref{section:smallnorm}, we perform an asymptotic analysis of $\Phi$ by means of the Deift/Zhou steepest descent method. This analysis is split into several sections as follows. In Section \ref{section:RH1}, we normalize the RH problem for $\Phi$ at infinity using an appropriate $\gg$-function, and then we proceed with the opening of the lenses. In Section \ref{section: global parametrix}, we construct the global parametrix $P^{(\infty)}$ in terms of Riemann $\theta$-functions. In Section \ref{section: local parametrix}, we construct local parametrices at $z=-x_j$, $j=1,\ldots,2g+1$; these are built using the well-known solution of the Bessel model RH problem (see Appendix \ref{Section:Appendix}). In Section \ref{section:smallnorm}, we complete our RH analysis with a small norm analysis. Section \ref{section: proof of thms} relies on the results of Sections \ref{section:RH1}-\ref{section:smallnorm} and contains the proofs of Theorems \ref{thm: main result - DIZ analogue}-\ref{thm: main result - ergodic}. Birkhoff's ergodic theorem is used in the proof of Theorem \ref{thm: main result - ergodic}.

\section{Differential identity for $F$}\label{section:diffid}
We begin by establishing a relationship between the Fredholm determinant $F(r\vec{x})$ and the solution to a particular Riemann-Hilbert problem.  The goal of this section is to prove the following Proposition. 
\begin{proposition}\label{prop: diff id}
Let $0<x_1<\cdots<x_{2g+1}<+\infty$ and $\alpha>-1$ be fixed.  We have the identity
\begin{align}\label{diff id}
\partial_{r} \log F(r\vec{x}) = \tfrac{1}{2ir}\Phi_{1,12}(r) + \tfrac{1-4\alpha^{2}}{16r},
\end{align}
where $F(\vec{x})$ is defined in \eqref{generating}, $\Phi_{1}(r) = \Phi_{1}(r;\vec{x},\alpha)$ is defined by
\begin{align}\label{def of Phi1}
\Phi_{1}(r) =  \lim_{z \to \infty}  rz \Big(\Phi(z)\big((rz)^{-\frac{\sigma_{3}}{4}}Me^{\sqrt{rz}\sigma_{3}}\big)^{-1} - I\Big), \quad M = \frac{1}{\sqrt{2}} \begin{pmatrix}
1 & i \\ i & 1
\end{pmatrix}, \quad \sigma_{3} = \begin{pmatrix}
1 & 0 \\ 0 & -1
\end{pmatrix},
\end{align}
and $\Phi(\cdot) = \Phi(\cdot;r, \vec{x},\alpha)$ is the unique solution of the following RH problem.
\end{proposition}
\subsubsection*{RH problem for $\Phi(\cdot) = \Phi(\cdot;r, \vec{x},\alpha)$}
\begin{itemize}
\item[(a)] $\Phi : \mathbb{C}\setminus \Sigma_{\Phi} \to \mathbb{C}^{2\times 2}$ is analytic, where the contour $\Sigma_{\Phi} = (-\infty,0]\cup \Sigma_+ \cup \Sigma_-$ is oriented as shown in Figure \ref{fig:modelRHcontours} with
\begin{equation*}
\Sigma_+ := -x_{2g+1}+ e^{\frac{2\pi i}{3}}(0,+\infty), \qquad \Sigma_- := -x_{2g+1}+ e^{-\frac{2\pi i}{3}}(0,+\infty).
\end{equation*}
\item[(b)] The limits of $\Phi(z)$ as $z$ approaches $\Sigma_{\Phi}\setminus \{-x_{2g+1},-x_{2g},\dots,-x_1,0\}$ from the left ($+$ side) and from the right ($-$ side) exist, are continuous on $\Sigma_{\Phi}\setminus \{-x_{2g+1},-x_{2g},\dots,-x_1,0\}$ and are denoted by $\Phi_+$ and $\Phi_-$ respectively. Furthermore, they are related by
\begin{align*}
& \Phi_{+}(z) = \Phi_{-}(z) \begin{pmatrix}
1 & 0 \\ e^{\pm \pi i\alpha} & 1
\end{pmatrix}, & & z \in \Sigma_\pm, \\
& \Phi_{+}(z) = \Phi_{-}(z) \begin{pmatrix}
0 & 1 \\ -1 & 0
\end{pmatrix}, & & z \in (-\infty,-x_{2g+1}), \\
& \Phi_{+}(z) = \Phi_{-}(z) \begin{pmatrix}
e^{\pi i\alpha} & 1 \\ 0 & e^{-\pi i\alpha}
\end{pmatrix}, & & z \in (-x_{2j},-x_{2j-1}), ~~~ j=1,\dots,g, \\
& \Phi_{+}(z) = \Phi_{-}(z) e^{\pi i\alpha\sigma_3}, & & z \in (-x_{2j+1},-x_{2j}), ~~~j=0,\dots,g,
\end{align*}
where $x_{0}:=0$.
\item[(c)] As $z \to \infty$, we have 
\begin{equation}\label{Phi inf}
\Phi(z) = \Big(I + \bigO(z^{-1})\Big)(rz)^{-\frac{\sigma_{3}}{4}}Me^{\sqrt{rz}\sigma_{3}},
\end{equation}
where the principal branch is chosen for each fractional power.

\item [(d)] As $z \to -x_{j}$, $j=1,\dots,2g+1$, we have 
\begin{equation}\label{def of Gj}
\Phi(z) = G_{j}(z) \begin{pmatrix}
1 & \frac{(-1)^{j+1}}{2\pi i} \log (rz+rx_{j}) \\ 0 & 1
\end{pmatrix} V_{j}(z) e^{\frac{\pi i\alpha}{2}\mathrm{sgn}(\im z)\sigma_{3}}H(z),
\end{equation}
where $\mathrm{sgn}$ is the signum function, $G_{j}$ is analytic in a neighborhood of $-x_{j}$ and satisfies $\det G_{j} \equiv 1$, $V_{j}(z)$ is defined by
\begin{align}
V_j(z)=I, \quad \text{if $j$ is odd}, \qquad\qquad V_{j}(z) = \left\{  \begin{array}{l l}
I, & \mbox{Im } z > 0, \\
\begin{pmatrix}
1 & -1 \\ 0 & 1
\end{pmatrix}, & \mbox{Im } z < 0,
\end{array} \right. \quad \text{if $j$ is even,}
\end{align}
and $H(z)$ is defined by
\begin{align}\label{def of H}
H(z)=\begin{cases}
I, & \text{for } -\frac{2\pi}{3}<\arg(z+x_{2g+1})<\frac{2\pi}{3}, \\
\begin{pmatrix} 1 & 0 \\ -e^{\pi i\alpha} & 1 \end{pmatrix} & \text{for } \frac{2\pi}{3}<\arg(z+x_{2g+1})<\pi, \\
\begin{pmatrix} 1 & 0 \\ e^{-\pi i\alpha} & 1 \end{pmatrix} & \text{for } -\pi<\arg(z+x_{2g+1})<-\frac{2\pi}{3}.
\end{cases}
\end{align}
\item[(e)] As $z$ tends to $0$, $\Phi$ takes the form
\begin{equation}\label{def of G_0}
\Phi(z) = G_{0}(z)(rz)^{\frac{\alpha}{2}\sigma_{3}}, \qquad \alpha > -1,
\end{equation}
where $G_{0}$ is analytic in a neighborhood of $0$ and satisfies $\det G_{0} \equiv 1$. 
\end{itemize}

\begin{figure}
\centering
\begin{tikzpicture}
\draw[fill] (0,0) circle (0.05);
\draw (0,0) -- (10,0);
\draw (0,0) -- (120:3);
\draw (0,0) -- (-120:3);
\draw (0,0) -- (-3,0);
\draw[fill] (1.6,0) circle (0.05);
\draw[fill] (4,0) circle (0.05);
\draw[fill] (5.4,0) circle (0.05);
\draw[fill] (8.4,0) circle (0.05);
\draw[fill] (10,0) circle (0.05);

\node at (0.2,-0.3) {$-x_{5}$};
\node at (1.6,-0.3) {$-x_{4}$};
\node at (4,-0.3) {$-x_{3}$};
\node at (5.4,-0.3) {$-x_{2}$};
\node at (8.4,-0.3) {$-x_{1}$};
\node at (10,-0.3) {$0$};

\node at (96:2) {$\begin{pmatrix} 1 & 0 \\ e^{\pi i \alpha} & 1 \end{pmatrix}$};
\node at (160:2) {$\begin{pmatrix} 0 & 1 \\ -1 & 0 \end{pmatrix}$};
\node at (-93:2) {$\begin{pmatrix} 1 & 0 \\ e^{-\pi i \alpha} & 1 \end{pmatrix}$};

\node at (2.8,0.6) {$\begin{pmatrix} e^{\pi i \alpha} & 1 \\ 0 & e^{-\pi i \alpha} \end{pmatrix}$};
\node at (6.9,0.6) {$\begin{pmatrix} e^{\pi i \alpha} & 1 \\ 0 & e^{-\pi i \alpha} \end{pmatrix}$};

\node at (0.8,0.3) {$e^{\pi i \alpha \sigma_{3}}$};
\node at (4.7,0.3) {$e^{\pi i \alpha \sigma_{3}}$};
\node at (9.2,0.3) {$e^{\pi i \alpha \sigma_{3}}$};

\draw[black,arrows={-Triangle[length=0.18cm,width=0.12cm]}]
(-120:1.5) --  ++(60:0.001);
\draw[black,arrows={-Triangle[length=0.18cm,width=0.12cm]}]
(120:1.3) --  ++(-60:0.001);
\draw[black,arrows={-Triangle[length=0.18cm,width=0.12cm]}]
(180:1.5) --  ++(0:0.001);

\draw[black,arrows={-Triangle[length=0.18cm,width=0.12cm]}]
(0:0.9) --  ++(0:0.001);
\draw[black,arrows={-Triangle[length=0.18cm,width=0.12cm]}]
(0:2.9) --  ++(0:0.001);
\draw[black,arrows={-Triangle[length=0.18cm,width=0.12cm]}]
(0:4.8) --  ++(0:0.001);
\draw[black,arrows={-Triangle[length=0.18cm,width=0.12cm]}]
(0:7) --  ++(0:0.001);
\draw[black,arrows={-Triangle[length=0.18cm,width=0.12cm]}]
(0:9.3) --  ++(0:0.001);

\end{tikzpicture}
\caption{The jump contour $\Sigma_{\Phi}$ and the associated jump matrices in the case when $g=2$.}
\label{fig:modelRHcontours}
\end{figure}
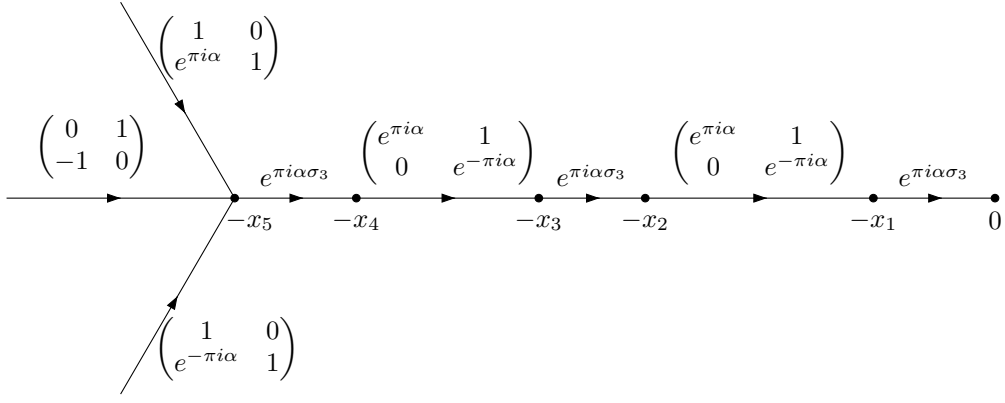
\begin{proof}[Proof of Proposition \ref{prop: diff id}]
We follow the method developed by Its, Izergin, Korepin and Slavnov \cite{IIKS}, and then further pursued in \cite{DIZ}. Note from \eqref{Bessel kernel} that $K$ can be rewritten as
\begin{align*}
K(x,y)=\frac{\vec{f}^{\hspace{0.6mm}t}(x)\vec{g}(y)}{x-y}, \qquad \vec{f}(x) := \frac{1}{2}\begin{pmatrix}
J_{\alpha}(\sqrt{x}) \\
\sqrt{x} J_{\alpha}^{\prime}(\sqrt{x})
\end{pmatrix}, \qquad \vec{g}(y) := \begin{pmatrix}
\sqrt{y} J_{\alpha}^{\prime}(\sqrt{y}) \\
-J_{\alpha}(\sqrt{y})
\end{pmatrix}.
\end{align*} 
Also, from \eqref{generating} and a direct rescaling, we get
\begin{align*}
F(r\vec{x}) = \det(I-\mathcal{K}|_{r\mathcal{I}_g}) = \det (I-\mathcal{K}_{r}|_{\mathcal{I}_g}),
\end{align*}
where $\mathcal{K}_{r}|_{\mathcal{I}_g}$ is the trace-class integral operator acting on $L^{2}(\mathcal{I}_{g})$ whose kernel is given by
\begin{align*}
K_{r}(x,y) := rK(rx,ry) = \frac{\vec{f}^{\hspace{0.6mm}t}(rx)\vec{g}(ry)}{x-y}.
\end{align*}
Using well-known identities for trace-class integral operators, we get
\begin{align}\label{lol2}
\partial_{r} \log F(r\vec{x})=\partial_{r} \log \det \big( I- \mathcal{K}_r|_{\mathcal{I}_g} \big)=\partial_{r} \text{Tr}\log \big( I- \mathcal{K}_r|_{\mathcal{I}_g} \big)=-\text{Tr}\big(\big( I- \mathcal{K}_r|_{\mathcal{I}_g} \big)^{-1}\partial_r\mathcal{K}_r|_{\mathcal{I}_g}\big).
\end{align}
A direct computation using \cite[eq. (8.401)]{GRtable} gives
\begin{align*}
\partial_{r}K_{r}(x,y) = \frac{1}{4}J_{\alpha}(\sqrt{rx})J_{\alpha}(\sqrt{ry}) = -\frac{1}{2}\vec{f}^{\hspace{0.6mm}t}(r x) \begin{pmatrix}
0 & 1 \\ 0 & 0
\end{pmatrix} \vec{g}(r y),
\end{align*}
and it follows that the kernel of the integral operator $\big( I- \mathcal{K}_r|_{\mathcal{I}_g} \big)^{-1}\partial_r\mathcal{K}_r|_{\mathcal{I}_g}$ is given by 
\begin{align*}
-\frac{1}{2}\vec{h}^{\hspace{0.1mm}t}(rx)\begin{pmatrix}
0 & 1 \\ 0 & 0
\end{pmatrix} \vec{g}(r y)=-\frac{1}{2}h_1(rx)g_2(ry),
\end{align*}
where $\vec{h}(\mu):=\Big((I-\mathcal{K}|_{r\mathcal{I}_g})^{-1} \vec{f}_{1}(\mu),(I-\mathcal{K}|_{r\mathcal{I}_g})^{-1} \vec{f}_{2}(\mu)\Big)^{t}$. Thus, by \eqref{lol2} we have
\begin{align*}
\partial_{r} \log \det \big( I- \mathcal{K}|_{r\mathcal{I}_g} \big)=\partial_{r} \log \det \big( I- \mathcal{K}_r|_{\mathcal{I}_g} \big)=\frac{1}{2}\int_{\mathcal{I}_g}h_1(rx)g_2(rx)dx=\frac{1}{2r}\int_{r\mathcal{I}_g}h_1(x)g_2(x)dx.
\end{align*}
On the other hand,
\begin{align*}
Y(z):=I-\int_{r\mathcal{I}_g}\frac{\vec{h}(x)\vec{g}^{\hspace{0.6mm}t}(x)}{x-z}dx=I+\frac{Y_1(r)}{z}+\bigO(z^{-2}) \qquad \mbox{as } z \to \infty
\end{align*}
where
\begin{align*}
Y_1(r)=\int_{r\mathcal{I}_g}\begin{pmatrix} h_1(x)g_1(x) & h_1(x)g_2(x) \\ h_2(x)g_1(x) & h_2(x)g_2(x) \end{pmatrix}dx,
\end{align*}
and so we conclude that
\begin{align}\label{logdetY112}
\partial_{r}\log F(r\vec{x})=\partial_{r} \log \det \big( I- \mathcal{K}|_{r\mathcal{I}_g} \big)=\frac{Y_{1,12}(r)}{2r}.
\end{align}
It follows from \cite[eq. (4.17)]{ChDoe} that we can express $Y$ in terms of $\Phi$ as follows
\begin{equation*}
\Phi(z) = \begin{pmatrix}
1 & 0 \\ \frac{i}{8}(4\alpha^{2}+3) & 1
\end{pmatrix} e^{\frac{\pi i }{4}\sigma_{3}} \sigma_{3}Y(-rz)\sigma_{3} \widetilde{P}_{\mathrm{Be}}(rz),
\end{equation*}
where $\widetilde{P}_{\mathrm{Be}}$, which is defined in \cite[eq. (A.7)]{ChDoe}, satisfies
\begin{equation*}
\widetilde{P}_{\mathrm{Be}}( z ) =  e^{-\frac{\pi i}{4}\sigma_{3}}
\begin{pmatrix}
1 & 0 \\ -\frac{i}{8}(4\alpha^{2}+3) & 1
\end{pmatrix}\left(I+\frac{\widetilde{P}_{\Be,1}}{z}+\bigO( z ^{-2})\right)  z ^{-\frac{\sigma_{3}}{4}} M e^{ z ^{\frac{1}{2}}\sigma_{3}} \qquad \mbox{as } z \to \infty,
\end{equation*}
with
\begin{equation}\label{PBe112}
    \widetilde{P}_{\Be,1,12}=\frac{1-4\alpha^{2}}{8i}.
\end{equation}
Using \eqref{Phi inf}, we infer that
\begin{align*}
\Phi_{1}(r) = \begin{pmatrix}
1 & 0 \\
\frac{i}{8}(4\alpha^{2}+3) & 1
\end{pmatrix}e^{\frac{ \pi i }{4}\sigma_{3}}\Big(-\sigma_{3} Y_{1}(r)\sigma_{3} \Big)e^{-\frac{\pi i}{4}\sigma_{3}} \begin{pmatrix}
1 & 0 \\
-\frac{i}{8}(4\alpha^{2}+3) & 1
\end{pmatrix} + \widetilde{P}_{\Be,1},
\end{align*}
and thus
\begin{align}\label{phi112 preliminary}
\Phi_{1,12}(r) = i Y_{1,12}(r) + \widetilde{P}_{\Be,1,12}.
\end{align}
The result now follows by using \eqref{PBe112} and \eqref{phi112 preliminary} in \eqref{logdetY112}.
\end{proof}


Our next goal is to obtain large $r$ asymptotics of $\Phi_{1,12}(r)$. For this, we will first perform a Deift--Zhou \cite{DeiftZhou} steepest descent analysis on the RH problem for $\Phi$. This analysis is carried out in Sections \ref{section:RH1}--\ref{section:smallnorm}.

\section{Steepest descent for $\Phi$: first steps}\label{section:RH1}
In this section, we proceed with the first steps of the steepest descent method.  We start by finding a $\gg$-function that will be used to normalize the RH problem for $\Phi$ at $\infty$.

\subsection{$\gg$-function}\label{subsection: g-function}
The function $\sqrt{\mathcal{R}(z)}$, defined in \eqref{intro: R}, satisfies
\begin{align}\label{jump for R 2cuts}
& \sqrt{\mathcal{R}(z)}_{+} + \sqrt{\mathcal{R}(z)}_{-} = 0, \qquad \mbox{for } z \in \bigcup_{j=1}^{g+1}(-x_{2j},-x_{2j-1}),
\end{align}
where $\sqrt{\mathcal{R}(z)}_{+}$ and $\sqrt{\mathcal{R}(z)}_{-}$ denote the limits of  $\sqrt{\mathcal{R}(s)}$ as $s \to z$ from the upper and lower half-plane of the first sheet of $X$, respectively. We define the $\gg$-function by
\begin{align}\label{gprime}
\gg(z) = \int_{-x_{1}}^{z^{+}} \frac{q(s)}{\sqrt{\mathcal{R}(s)}}ds, \qquad q(z)=\frac{z^g}{2}+\sum_{j=0}^{g-1}q_jz^j, \qquad z \in \mathbb{C}\setminus (-\infty,-x_{1}),
\end{align}
where the path of integration lies on the first sheet of $X$ and does not cross $(-\infty,-x_{1}]^{+}$ and coefficients $q_j$, $j=0,\dots,g-1$ are given by \eqref{intro: q}.  The following lemma extracts the necessary properties of the $\gg$-function.

\begin{lemma}\label{g_lemma} 
\begin{enumerate}
\, The $\gg$-function enjoys the following properties:
\item  The $\gg$-function is analytic in $\mathbb{C}\setminus(-\infty,-x_1]$ and satisfies $\gg(z) = \overline{\gg(\overline{z})}$.
\item The $\gg$-function satisfies the jump conditions 
\begin{align}
& \gg_{+}(z) + \gg_{-}(z) = 0, & & z \in (-x_{2j},-x_{2j-1}), \qquad j=1,\dots,g+1, \label{jump1 g 2cuts} \\
& \gg_+(z) - \gg_-(z) = i(\Omega_{1}+\cdots+\Omega_j), & & z \in (-x_{2j+1},-x_{2j}), \qquad j=1,\dots,g, \label{jump3 g 2cuts}
\end{align}
where $\Omega_{j}>0$ is defined in \eqref{intro: Omegaj}.
\item As $z\to\infty$, $z \notin (-\infty,-x_{2g+1})$, we have
\begin{equation}\label{asymp g 2cuts}
\gg(z) = \sqrt{z}-\frac{2c}{\sqrt{z}}+\bigO(z^{-3/2}),
\end{equation}
where $c$ is defined in \eqref{intro: c const}.
\item $\re \gg(z)\geq0$ for $z\in\mathbb{C}$ with equality only when $z\in[-x_{2j},-x_{2j-1}]$,  $j=1,\dots,g+1$, where $x_{2g+2}:=-\infty$.
\end{enumerate}
\end{lemma}

\begin{proof}

We begin with the observation that the second set of equations in \eqref{prop of omega and q in intro} can be rewritten as
\begin{align}\label{qInt conditions}
\int_{-x_{2j+1}}^{-x_{2j}}\frac{q(s)}{\sqrt{\mathcal{R}(s)}}ds=0, \qquad j=1,\dots,g,
\end{align}
where the path lies on the first sheet of $X$. Because $\sqrt{\mathcal{R}(z)}$ is real and non-zero on $(-x_{2j+1},-x_{2j})^{+}$, it follows from \eqref{qInt conditions} that $q$ has exactly one simple zero on each of the intervals $(-x_{2g+1},-x_{2g}), \ldots$, $(-x_{5},-x_{4})$, $(-x_{3},-x_{2})$. Since $q(z)\to+\infty$ as $z\to+\infty$, it satisfies
\begin{align}\label{sgn q}
\text{sgn} \; q(z)&=(-1)^{j+1}, \qquad \mbox{ for all }z\in(-x_{2j},-x_{2j-1}), \quad j=1,\dots,g+1.
\end{align}

\begin{enumerate}
\item The fact that $\gg$ is analytic in $\mathbb{C}\setminus(-\infty,-x_1]$ follows directly from \eqref{gprime}, and $\overline{\gg(\overline{z})}= \gg(z)$ follows from the fact that the coefficients of $\mathcal{R}$ and $q$ are real.

\item The jumps \eqref{jump1 g 2cuts} follow from \eqref{jump for R 2cuts} and \eqref{qInt conditions}. 
For $z \in (-x_{2j+1},-x_{2j})$, by \eqref{gprime} we have
\begin{align*}
\gg_{+}(z) - \gg_{-}(z) = \sum_{k=1}^{j}\int_{B_{k}} \gg'(s)ds = i (\Omega_{1}+\cdots+\Omega_{j}),
\end{align*}
and this is \eqref{jump3 g 2cuts}.  Using \eqref{sgn q} and the definition of $\Omega_{j}$ in \eqref{intro: Omegaj}, it follows that $\Omega_j > 0$.

\item A direct computation using \eqref{gprime} gives $\gg'(z) = \frac{1}{2}z^{-\frac{1}{2}}+c \, z^{-\frac{3}{2}}+\bigO(z^{-\frac{5}{2}})$ as $z \to \infty$. Hence,
\begin{align}\label{lol1}
\gg(z) = \sqrt{z}+g_0-\frac{2c}{\sqrt{z}}+\bigO(z^{-\frac{3}{2}}) \qquad \mbox{as } z \to \infty,
\end{align}
for a certain $g_{0} \in \mathbb{C}$. We know from \eqref{jump1 g 2cuts} that $\gg_{+}(z) + \gg_{-}(z) = 0$, which implies $g_{0} = 0$.

\item It follows from the zeros of $q(z)$ and \eqref{sgn q} that $\re \gg(z) >0$ for $z \in (-x_{1},+\infty)$,
\begin{align}\label{real part of g on R inside proof}
\re \gg(z) =0 \quad \mbox{for } z\in \bigcup_{j=1}^{g+1}[-x_{2j},-x_{2j-1}] \quad \mbox{ and } \quad \re \gg(z) >0 \quad \mbox{for } z \in \bigcup_{j=1}^{g}(-x_{2j+1},-x_{2j}),
\end{align}
It follows from the definition \eqref{gprime} that $\gg$ admits an expansion of the form $\gg(z)\sim \sqrt{z}(1+\sum_{j=1}^\infty\frac{\gg_j}{z^j})$ as $z \to \infty$, and since $\overline{\gg(\overline{z})}= \gg(z)$, we infer that $\gg_j \in \mathbb{R}$ for all $j\geq 1$. Hence, as $z = |z|e^{i\theta} \to \infty$, $|\theta| \leq \pi$, we have
\begin{align*}
\re \gg(z) \sim \sqrt{|z|}\cos \frac{\theta}{2} \bigg( 1 + \sum_{j=1}^{\infty} \frac{\gg_{j}}{|z|^{j}} \frac{\cos \big( (j-\frac{1}{2})\theta \big)}{\cos \frac{\theta}{2}} \bigg) \qquad \mbox{as } z \to \infty.
\end{align*}
Since 
\begin{align*}
\left| \frac{\cos \big( (j-\frac{1}{2})\theta \big)}{\cos \frac{\theta}{2}} \right| \leq 2j-1, \qquad \mbox{for all } |\theta| \leq \pi  \mbox{ and } j=1,2,\ldots
\end{align*}
there exist a constant $\rho_*>0$ such that 
\begin{align*}
\left| \frac{\re \gg(|z|e^{i\theta})}{ \sqrt{|z|}\cos \frac{\theta}{2}}- 1\right| \leq \frac{1}{2}, \qquad \mbox{for all } |z|\geq \rho_{*} \mbox{ and } |\theta|\leq \pi.
\end{align*}
This implies that 
\begin{align}\label{lol10}
\frac{1}{2}\sqrt{|z|}\cos \frac{\theta}{2} \leq \re \gg(|z|e^{i\theta}) \leq \frac{3}{2}\sqrt{|z|}\cos \frac{\theta}{2} \qquad \mbox{for all } |z|\geq \rho_{*} \mbox{ and } |\theta|\leq \pi.
\end{align}
Let $\mathcal{U}_{s}$ be the open half disk in the upper half-plane centered at the origin with radius $s \geq \rho_{*}$. Since $\re \gg(z)$ is harmonic on $\mathcal{U}_{s}$, its minimum value over the set $\overline{\mathcal{U}_{s}}$ is attained only on $\partial\mathcal{U}_{s}$, which by \eqref{real part of g on R inside proof} and \eqref{lol10} is equal to $0$. Thus $\re \gg(z)>0$ for all $z\in\mathcal{U}_{s}$ Since $s \geq \rho_{*}$ was arbitrary, we conclude that $\re \gg(z)>0$ for all $z$ in the upper half-plane. The symmetry $\overline{\gg(\overline{z})}= \gg(z)$ now implies that $\re \gg(z)>0$ for all $z$ in the lower half-plane, which finishes the proof.
\end{enumerate}
\end{proof}

\subsection{Rescaling of the RH problem}\label{Subsection: Phi to T}

We now use the $\gg$-function \eqref{gprime} to define
\begin{align}\label{def of T}
    T(z):=\begin{pmatrix} 1 & 0 \\ -2ic\sqrt{r} & 1 \end{pmatrix}r^\frac{\sigma_3}{4}\Phi(z)e^{-\sqrt{r}\gg(z)\sigma_3}.
\end{align}
It can now be verified, using Lemma \ref{g_lemma} and the RH problem for $\Phi$, that $T$ is the unique solution to the following RH problem.

\subsubsection*{RH problem for $T(\cdot)=T(\cdot;r, \vec{x},\alpha)$}
\begin{enumerate}[label={(\alph*)}]
\item[(a)] $T : \C \backslash \Sigma_{\Phi} \rightarrow \C^{2\times 2}$ is analytic, and we recall that $\Sigma_{\Phi}$ is shown in Figure \ref{fig:modelRHcontours}.
\item[(b)] The jumps for $T$ are given by
\begin{align*}
& T_{+}(z) = T_{-}(z)\begin{pmatrix}
0 & 1 \\ -1 & 0
\end{pmatrix}, & & z \in (-\infty,-x_{2g+1}), \\
& T_{+}(z) = T_{-}(z)\begin{pmatrix}
1 & 0 \\ e^{\pm i\pi\alpha-2\sqrt{r}\gg(z)} & 1
\end{pmatrix}, & & z \in \Sigma_\pm, \\
& T_{+}(z) = T_{-}(z)e^{(i\pi\alpha-i\widehat{\Omega}_{j}\sqrt{r})\sigma_3}, & & z \in (-x_{2j+1},-x_{2j}), \quad j=0,\dots,g,
\end{align*}
where we define 
\begin{align*}
\widehat{\Omega}_{0}=0 \qquad \mbox{ and } \qquad \widehat{\Omega}_{j}=\Omega_1+\cdots+\Omega_j, \quad j=1,\ldots,g.
\end{align*}
For $z\in \bigcup_{j=1}^{g}(-x_{2j},-x_{2j-1})$ we have
\begin{align}\label{3factorization}
& T_{+}(z) = T_{-}(z)\begin{pmatrix} 1 & 0 \\ e^{-i\pi\alpha-2\sqrt{r}\gg_-(z)} & 1 \end{pmatrix}\begin{pmatrix} 0 & 1 \\ -1 & 0 \end{pmatrix}\begin{pmatrix} 1 & 0 \\ e^{i\pi\alpha-2\sqrt{r}\gg_+(z)} & 1 \end{pmatrix}.
\end{align}
\item[(c)] As $z \rightarrow \infty$, we have
\begin{equation}
\label{eq:Tasympinf}
T(z) = \left( I + \frac{T_1}{z}+ \Or\left(z^{-2}\right) \right) z^{-\frac{\sigma_3}{4}}M,
\end{equation}
where $T_{1}$ is independent of $z$ and satisfies $T_{1,12}=\frac{\Phi_{1,12}}{\sqrt{r}}-2ic\sqrt{r}$.
\item[(d)] $T(z) = \Or( \log(z+x_j) )$ as $z \to -x_{j}$, $j=1,\dots,2g+1$.
\item[(e)] As $z\to0$, $T$ takes the form
\begin{equation}\label{T0}
    T(z)=T_0(z)z^{\frac{\alpha}{2}\sigma_3},
\end{equation}
where $T_0$ is analytic in a neighborhood of $0$.
\end{enumerate}

\subsection{Opening of the lenses}\label{Subsection: T to S}
For each $j \in \{1,\ldots,g\}$, let $\Sigma_{j,+}$ and $\Sigma_{j,-}$ be two open curves starting at $-x_{2j}$, ending at $-x_{2j-1}$, and lying in the upper and lower half-plane respectively, see also Figure \ref{fig:contour for S}. The bounded lens-shaped region delimited by $\Sigma_{j,+}\cup \Sigma_{j,-}$ will be denoted by $\mathcal{L}_{j}$, $j=1,\ldots,g$. With the factorization \eqref{3factorization} in mind, we define $S$ in terms of $T$ as follows
\begin{align}\label{def:S}
\hspace{-0.2cm}S(z):=T(z) \left\{ \hspace{-0.1cm} \begin{array}{l l}
\begin{pmatrix}
1 & 0 \\
-e^{\pi i\alpha-2\sqrt{r}\gg(z)} & 1
\end{pmatrix}, & z \in \cup_{j=1}^{g}\mathcal{L}_{j} \mbox{ and } \im z > 0, \\
\begin{pmatrix}
1 & 0 \\
e^{-\pi i\alpha-2\sqrt{r}\gg(z)} & 1
\end{pmatrix}, & z \in \cup_{j=1}^{g}\mathcal{L}_{j} \mbox{ and } \im z< 0, \\
I, & \mbox{otherwise}.
\end{array} \right.
\end{align}
It is straightforward to verify that $S$ satisfies the following RH problem.

\subsubsection*{RH problem for $S(\cdot)=S(\cdot;r, \vec{x},\alpha)$}
\begin{enumerate}[label={(\alph*)}]
\item[(a)] $S : \C \setminus \Sigma_{S} \rightarrow \C^{2\times 2}$ is analytic, with
\begin{equation}\label{def:SigmaS}
\Sigma_{S}=(-\infty,0]\cup\gamma_{+}\cup\gamma_{-}, \qquad \gamma_{\pm} = \Sigma_\pm\cup\Sigma_{1,\pm}\cup\cdots\cup\Sigma_{g,\pm},
\end{equation}
and $\Sigma_{S}$ is oriented as in Figure \ref{fig:contour for S}.
\item[(b)] The boundary values of $S$ are related by
\begin{align}
& S_{+}(z) = S_{-}(z)\begin{pmatrix}
0 & 1 \\ -1 & 0
\end{pmatrix}, & & z \in (-x_{2j},-x_{2j-1}), ~ j=1,\dots,g+1, \label{SInf0 jump} \\
& S_{+}(z) = S_{-}(z)\begin{pmatrix}
1 & 0 \\ e^{\pm \pi i\alpha-2\sqrt{r}\gg(z)} & 1
\end{pmatrix}, & & z \in \gamma_{\pm}, \\
& S_{+}(z) = S_{-}(z)e^{(\pi i\alpha-i\widehat{\Omega}_{j}\sqrt{r})\sigma_{3}}, & & z \in (-x_{2j+1},-x_{2j}), ~ j=0,\dots,g, \label{S0y2 jump}
\end{align}
where we recall that $x_{2g+2}:=-\infty$.
\item[(c)] As $z \rightarrow \infty$, we have
\begin{equation}
\label{eq:Sasympinf}
S(z) = \left( I +\frac{T_1}{z}+\Or\left(z^{-2}\right) \right) z^{-\frac{\sigma_3}{4}}M.
\end{equation}
\item[(d)] As $z \to -x_j$, $j=1,\dots,2g+1$, we have
\begin{align}\label{S asymp at xj}
S(z) = \Or( \log(z+x_j) ).
\end{align}
\item[(e)] As $z\to0$, $S$ takes the form
\begin{equation}\label{S0}
    S(z)=S_0(z)z^{\frac{\alpha}{2}\sigma_3},
\end{equation}
where $S_0$ is analytic in a neighborhood of $0$.
\end{enumerate}
Recall from Lemma \ref{g_lemma} that 
\begin{align*}
\re \gg(z)>0 \mbox{ for all }z\in \mathbb{C}\setminus \bigcup_{j=1}^{g+1}[-x_{2j},-x_{2j-1}] \quad \mbox{ and } \quad \re \gg(z)=0 \mbox{ for all } z\in \bigcup_{j=1}^{g+1}[-x_{2j},-x_{2j-1}].
\end{align*}
Hence, for any $\epsilon > 0$, there exists $c_{1}>0$ such that 
\begin{align*}
& S_{-}(z)^{-1}S_{+}(z)-I = \bigO(e^{-c_{1}r}) \mbox{ as } r \to + \infty \mbox{ uniformly for } z\in (\gamma_{+}\cup\gamma_{-}) \cap \Big\{z:\min_{j=1,\ldots,2g+1}\{|z+x_{j}|\}>\epsilon\Big\}.
\end{align*}
However, $S_{-}(z)^{-1}S_{+}(z)-I$ is not exponentially small as $r  \to + \infty$ when simultaneously $z \in \gamma_{+}\cup \gamma_{-}$, $z \to -x_{j}$ for $j = 1,\ldots,2g+1$.
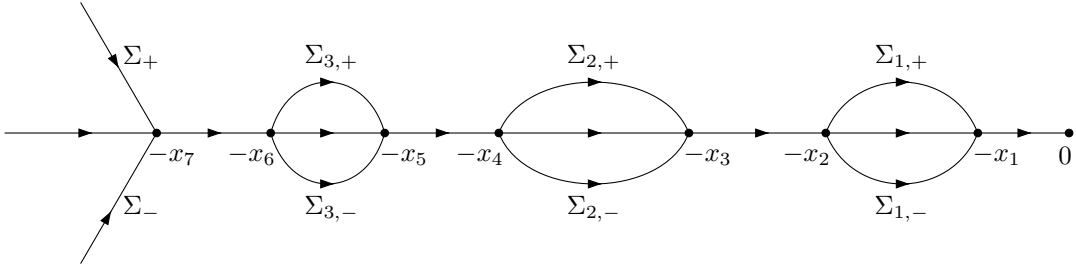
\begin{figure}
\centering
\begin{tikzpicture}

\draw[fill] (0,0) circle (0.05);
\draw[fill] (1.5,0) circle (0.05);
\draw[fill] (3,0) circle (0.05);
\draw[fill] (4.5,0) circle (0.05);
\draw[fill] (7,0) circle (0.05);
\draw[fill] (8.8,0) circle (0.05);
\draw[fill] (10.8,0) circle (0.05);
\draw[fill] (12,0) circle (0.05);
\draw (-2,0) -- (12,0);
\draw (0,0) -- (120:2);
\draw (0,0) -- (-120:2);

\draw (1.5,0) .. controls (1.75,0.9) and (2.75,0.9) .. (3,0);
\draw (1.5,0) .. controls (1.75,-0.9) and (2.75,-0.9) .. (3,0);

\draw (4.5,0) .. controls (4.85,0.9) and (6.65,0.9) .. (7,0);
\draw (4.5,0) .. controls (4.85,-0.9) and (6.65,-0.9) .. (7,0);

\draw (8.8,0) .. controls (9.15,0.9) and (10.45,0.9) .. (10.8,0);
\draw (8.8,0) .. controls (9.15,-0.9) and (10.45,-0.9) .. (10.8,0);

\node at (0.2,-0.3) {$-x_7$};
\node at (1.25,-0.3) {$-x_6$};
\node at (3.25,-0.3) {$-x_5$};
\node at (4.25,-0.3) {$-x_4$};
\node at (7.25,-0.3) {$-x_3$};
\node at (8.55,-0.3) {$-x_2$};
\node at (11.05,-0.3) {$-x_1$};
\node at (11.95,-0.3) {$0$};

\node at (-0.2,1) {$\Sigma_+$};
\node at (-0.2,-1) {$\Sigma_-$};
\node at (2.3,1) {$\Sigma_{3,+}$};
\node at (2.3,-1) {$\Sigma_{3,-}$};
\node at (5.75,1) {$\Sigma_{2,+}$};
\node at (5.75,-1) {$\Sigma_{2,-}$};
\node at (9.8,1) {$\Sigma_{1,+}$};
\node at (9.8,-1) {$\Sigma_{1,-}$};

\draw[black,arrows={-Triangle[length=0.18cm,width=0.12cm]}]
(-120:1.2) --  ++(60:0.001);
\draw[black,arrows={-Triangle[length=0.18cm,width=0.12cm]}]
(120:1.0) --  ++(-60:0.001);

\draw[black,arrows={-Triangle[length=0.18cm,width=0.12cm]}]
(0:-0.85) --  ++(0:0.001);
\draw[black,arrows={-Triangle[length=0.18cm,width=0.12cm]}]
(0:0.85) --  ++(0:0.001);
\draw[black,arrows={-Triangle[length=0.18cm,width=0.12cm]}]
(0:3.85) --  ++(0:0.001);
\draw[black,arrows={-Triangle[length=0.18cm,width=0.12cm]}]
(0:8.05) --  ++(0:0.001);
\draw[black,arrows={-Triangle[length=0.18cm,width=0.12cm]}]
(0:11.5) --  ++(0:0.001);

\draw[black,arrows={-Triangle[length=0.18cm,width=0.12cm]}]
(0:2.33) --  ++(0:0.001);
\draw[black,arrows={-Triangle[length=0.18cm,width=0.12cm]}]
(2.33,0.675) --  ++(0:0.001);
\draw[black,arrows={-Triangle[length=0.18cm,width=0.12cm]}]
(2.33,-0.675) --  ++(0:0.001);

\draw[black,arrows={-Triangle[length=0.18cm,width=0.12cm]}]
(0:5.85) --  ++(0:0.001);
\draw[black,arrows={-Triangle[length=0.18cm,width=0.12cm]}]
(5.85,0.675) --  ++(0:0.001);
\draw[black,arrows={-Triangle[length=0.18cm,width=0.12cm]}]
(5.85,-0.675) --  ++(0:0.001);

\draw[black,arrows={-Triangle[length=0.18cm,width=0.12cm]}]
(0:9.9) --  ++(0:0.001);
\draw[black,arrows={-Triangle[length=0.18cm,width=0.12cm]}]
(9.9,0.675) --  ++(0:0.001);
\draw[black,arrows={-Triangle[length=0.18cm,width=0.12cm]}]
(9.9,-0.675) --  ++(0:0.001);

\end{tikzpicture}
\caption{Jump contours $\Sigma_{S}$ for $S$ with $g=3$.}
\label{fig:contour for S}
\end{figure}

\section{Global parametrix}\label{section: global parametrix}
In this section we construct the global parametrix $P^{(\infty)}(z)$, which will serve as an approximation to $S(z)$ when $z$ is bounded away from the endpoints $-x_j$, $j=1,\dots,2g+1$. 
\subsubsection*{RH problem for $P^{(\infty)}$}
\begin{enumerate}[label={(\alph*)}]
\item[(a)] $P^{(\infty)} : \C \backslash (-\infty,0] \rightarrow \C^{2\times 2}$ is analytic.
\item[(b)] The jumps for $P^{(\infty)}$ are given by
\begin{align}
& P^{(\infty)}_{+}(z) = P^{(\infty)}_{-}(z)\begin{pmatrix}
0 & 1 \\ -1 & 0
\end{pmatrix}, & & z \in (-x_{2j},-x_{2j-1}), ~ j=1,\dots,g+1, \nonumber \\
& P^{(\infty)}_{+}(z) = P^{(\infty)}_{-}(z)e^{(\pi i\alpha-i\widehat{\Omega}_{j}\sqrt{r})\sigma_{3}}, & & z \in (-x_{2j+1},-x_{2j}), ~ j=0,\dots,g, \label{jumps for Pinf on (-x3,-x2)}
\end{align}
\item[(c)] As $z \rightarrow \infty$, we have
\begin{equation}
\label{eq:Pinf asympinf}
P^{(\infty)}(z) = \left( I + \frac{P^{(\infty)}_1}{z} + \Or\left(z^{-2}\right) \right) z^{-\frac{\sigma_3}{4}}M,
\end{equation}
where $P_{1}^{(\infty)}$ is independent of $z$.
\item[(d)] As $z \to -x_j$, $j=1,\dots,2g+1$, we have
\begin{equation}
    P^{(\infty)}(z) = \bigO((z+x_{j})^{-\frac{1}{4}}).
\end{equation}
\item[(e)] As $z\to0$, $P^{(\infty)}(z)$ takes the form
\begin{equation}\label{Pinf0}
    P^{(\infty)}(z)=P^{(\infty)}_0(z)z^{\frac{\alpha}{2}\sigma_3},
\end{equation}
where $P^{(\infty)}_0$ is analytic in a neighborhood of $0$.
\end{enumerate}
We will construct $P^{(\infty)}$ in terms of Riemann $\theta$-functions using some ideas from \cite{BCL2020, BCL2020 2, DIZ, KuijVanlessen}. However, because $P^{(\infty)}$ is not analytic on an unbounded cut, our construction is technically different from \cite{DIZ,KuijVanlessen}. As in \cite{KuijVanlessen}, we proceed in two steps. First, we construct the solution $\tilde{P}^{(\infty)}$ of an auxiliary RH problem which has no root-type singularity at the origin. Next, we consider a function $D$ with a root-type behavior at $0$ and suitable jumps. We will see that $P^{(\infty)}$ is expressed in terms of $\tilde{P}^{(\infty)}$ and $D$.

\subsection[]{Construction of $P^{(\infty)}$}
For the moment we let $\vec\nu = (\nu_{1},\ldots,\nu_{g})\in\mathbb{R}^g$ be arbitrary (we will require $\vec\nu$ to be as in \eqref{intro: nu} later). Consider the following RH problem:
\subsubsection*{RH problem for $\tilde{P}^{(\infty)}(\cdot)=\tilde{P}^{(\infty)}(\cdot;\vec{\nu})$}
\begin{enumerate}[label={(\alph*)}]
\item[(a)] $\tilde{P}^{(\infty)} : \C \backslash (-\infty,-x_1] \rightarrow \C^{2\times 2}$ is analytic.
\item[(b)] The jumps for $\tilde{P}^{(\infty)}$ are given by
\begin{align}
& \tilde{P}^{(\infty)}_{+}(z)=\tilde{P}^{(\infty)}_{-}(z)\begin{pmatrix}
0 & 1 \\ -1 & 0
\end{pmatrix}, & & z \in (-x_{2j},-x_{2j-1}), ~ j=1,\dots,g+1, \nonumber \\
& \tilde{P}^{(\infty)}_{+}(z) = \tilde{P}^{(\infty)}_{-}(z)e^{2\pi i(\nu_1+\cdots+\nu_j)\sigma_{3}}, & & z \in (-x_{2j+1},-x_{2j}), ~ j=1,\dots,g. \label{jumps for tildePinf on (-x3,-x2)}
\end{align}
\item[(c)] As $z \rightarrow \infty$, we have
\begin{equation}
\label{eq:tildePinf asympinf}
\tilde{P}^{(\infty)}(z) = \left( I + \frac{\tilde{P}^{(\infty)}_1}{z} + \Or\left(z^{-2}\right) \right) z^{-\frac{\sigma_3}{4}}M,
\end{equation}
where $\tilde{P}^{(\infty)}_1$ is independent of $z$.
\item[(d)] As $z \to -x_j$, $j=1,\dots,2g+1$, we have \begin{equation}
    \tilde{P}^{(\infty)}(z) = \bigO((z+x_j)^{-\frac{1}{4}}).
\end{equation}
\end{enumerate}

We express $\tilde{P}^{(\infty)}$ in terms of Riemann $\theta$-functions, whose definition and relevant properties are given in \eqref{intro: theta def} and \eqref{intro: ThetaPeriodic}. 
Recall from \eqref{def of Abel map} that the Abel map $\vec{\varphi}$ is given by
\begin{align}\label{AbelMap}
\vec{\varphi}: X \to J(X), \qquad P \mapsto \vec{\varphi}(P)=\int_{-x_1}^P \vec{\omega}^{t},
\end{align}
and that $\vec{\omega}$ is defined in \eqref{intro: omegaVec}. It will be convenient to also consider the following function, which we denote by $\vec{\varphi}^{\hspace{0.4mm} \mathbb{C}}$:
\begin{align}\label{varphi on the plane}
\vec{\varphi}^{\hspace{0.4mm} \mathbb{C}}: \mathbb{C} \setminus (-\infty,-x_{1}] \to \mathbb{C}^{g}, \qquad z \mapsto \vec{\varphi}^{\hspace{0.4mm} \mathbb{C}}(z)=\int_{-x_1}^{z^{+}} \vec{\omega}^{t},
\end{align}
where the path of integration lies on the first sheet. Recall that the period matrix $\tau$ is defined  by \eqref{intro: tau}. It is well-known, see e.g. \cite[p. 63]{FK1992}, that $\tau$ is symmetric and has positive definite imaginary part. The following properties of $\vec{\varphi}^{\hspace{0.4mm} \mathbb{C}}$ follow readily from its definition.

\begin{proposition}\label{AbelMapProperties}
For $z\in \mathbb{C}\setminus (-\infty,-x_{1}]$, we have
\begin{align}
\vec{\varphi}_+^{\hspace{0.4mm} \mathbb{C}}(z)+\vec{\varphi}_-^{\hspace{0.4mm} \mathbb{C}}(z)&=\vec{e}_1-\vec{e}_{j}, & & z\in(-x_{2j},-x_{2j-1}), ~ j=1,\dots,g+1, \label{AbelMap_JumpCut} \\
\vec{\varphi}_+^{\hspace{0.4mm} \mathbb{C}}(z)-\vec{\varphi}_-^{\hspace{0.4mm} \mathbb{C}}(z)&=\sum_{k=1}^j\vec{\tau}_k, & & z\in(-x_{2j+1},-x_{2j}), ~ j=1,\dots,g, \label{AbelMap_JumpGap}
    \end{align}
    where $\vec{e}_j$ denotes the $j$th column of the $g\times g$ identity matrix, $\vec{e}_{g+1}:=\vec{0}$, and $\vec{\tau}_j$ is the $j$th column of the matrix $\tau$.  Moreover, we have $\vec{\varphi}^{\hspace{0.4mm} \mathbb{C}}(-x_1)=\vec{0}$, $\vec{\varphi}^{\hspace{0.4mm} \mathbb{C}}(\infty)=\frac{1}{2}\vec{e}_1$, and, for $j=1,\dots,g$,
\begin{align}\label{AbelMapEndPts}
\vec{\varphi}_\pm^{\hspace{0.4mm} \mathbb{C}}(-x_{2j})&=\frac{1}{2}(\vec{e}_1-\vec{e}_j)\pm\frac{1}{2}\sum_{k=1}^j\vec{\tau}_k, ~~~ \vec{\varphi}_\pm^{\hspace{0.4mm} \mathbb{C}}(-x_{2j+1})=\frac{1}{2}(\vec{e}_1-\vec{e}_{j+1})\pm\frac{1}{2}\sum_{k=1}^j\vec{\tau}_k.
\end{align}
\end{proposition}

Define 
\begin{align*}
    \beta(z) = \sqrt[4]{\frac{(z+x_2)(z+x_4)\cdots(z+x_{2g})}{(z+x_1)(z+x_3)\cdots(z+x_{2g+1})}}, \qquad z \in \mathbb{C}\setminus \bigcup_{j=1}^{g+1}[x_{2j},x_{2j-1}],
\end{align*}
where the branch is chosen such that $\beta(z) > 0$ for $z > -x_1$, and define $\mathcal{G}:\mathbb{C}^g\to\mathbb{C}$ by
\begin{equation}\label{mathcalGdef}
\mathcal{G}(\vec z):=\frac{\theta(\vec z+\vec\nu)}{\theta(\vec z)}.
\end{equation}
For any $\vec{\lambda}, \vec{\lambda}'\in\mathbb{Z}^g$, $\mathcal{G}$ satisfies
\begin{align}\label{periodicity of G}
    \mathcal{G}(\vec{z}+\vec{\lambda}'+\tau\vec{\lambda})=e^{-2\pi i\vec{\lambda}^{t}\vec{\nu}}\mathcal{G}(\vec{z}).
\end{align}
This periodicity property follows immediately from \eqref{intro: ThetaPeriodic}. Since $-i\tau$ is real and positive definite, $\theta(\vec{u}) > 0$ for all $\vec{u} \in \mathbb{R}^{g}$, see \cite[Section 4]{DKMVZ2}. In particular, $\mathcal{G}(\vec{u}) > 0$ for all $\vec{u} \in \mathbb{R}^{g}$. Recall from \eqref{intro: omegaVec} that the holomorphic differentials $\omega_1,\ldots,\omega_{g}$ are given by
\begin{align}\label{omegaJ}
\omega_j = (1 \;\; z \;\; \cdots \; \; z^{g-1})\mathbb{A}^{-1}\vec{e}_{j}\frac{dz}{\sqrt{\mathcal{R}(z)}} = \sum_{k=1}^{g}(\mathbb{A}^{-1})_{kj}z^{k-1}\frac{dz}{\sqrt{\mathcal{R}(z)}}, \qquad j=1,\ldots,g.
\end{align}
It can now be verified that
\begin{align}\label{tildePinf}
\tilde{P}^{(\infty)}(z)=\tilde{Q}\frac{\beta(z)^{\sigma_3}}{\sqrt{2}}\begin{pmatrix}
    \ds \mathcal{G}(-\vec{\varphi}^{\hspace{0.4mm} \mathbb{C}}(z)+\tfrac{\vec{e}_1}{2}) & \ds i\mathcal{G}(\vec{\varphi}^{\hspace{0.4mm} \mathbb{C}}(z)+\tfrac{\vec{e}_1}{2}) \\
    \ds i\mathcal{G}(-\vec{\varphi}^{\hspace{0.4mm} \mathbb{C}}(z)) & \ds \mathcal{G}(\vec{\varphi}^{\hspace{0.4mm} \mathbb{C}}(z)) 
\end{pmatrix}, \qquad \tilde{Q}=\begin{pmatrix}
\ds \frac{1}{\mathcal{G}(\vec{0})} & 0 \\
\ds \frac{-ic_\mathcal{G}}{\mathcal{G}(\vec{0})} & \ds \frac{1}{\mathcal{G}(\frac{\vec{e}_1}{2})}
\end{pmatrix}
\end{align}
where 
\begin{equation}\label{cg const}
c_\mathcal{G}:=\frac{2(\mathbb{A}^{-1})_{g:}\nabla\mathcal{G}\left(\frac{\vec{e}_1}{2}\right)}{\mathcal{G}\left(\frac{\vec{e}_1}{2}\right)} = \frac{2}{\mathcal{G}\left(\frac{\vec{e}_1}{2}\right)}\sum_{j=1}^{g} (\mathbb{A}^{-1})_{gj} \boldsymbol{\partial}_{j}\mathcal{G}\big(\tfrac{\vec{e}_1}{2}\big),
\end{equation}
$(\mathbb{A}^{-1})_{g:}$ denotes the $g$-th row of $\mathbb{A}^{-1}$, $\nabla\mathcal{G}$ is the gradient of $\mathcal{G}$ (of size $g \times 1$), and $\boldsymbol{\partial}_{j}\mathcal{G}$ denotes the $j$-th partial derivative of $\mathcal{G}$. Moreover,
\begin{equation}\label{tildePinf112}
\tilde{P}^{(\infty)}_{1,12}=\frac{-2i(\mathbb{A}^{-1})_{g:}\nabla\mathcal{G}(\vec{0})}{\mathcal{G}(\vec{0})} = \frac{-2i}{\mathcal{G}(\vec{0})}\sum_{j=1}^{g} (\mathbb{A}^{-1})_{gj}\boldsymbol{\partial}_{j}\mathcal{G}(\vec{0}).
\end{equation}
Note that $P^{(\infty)}$ and $\tilde{P}^{(\infty)}$ have the same jumps on $(-x_{2j},-x_{2j-1})$, $j=1,\dots,g+1$ but not on $(-x_{2j+1},-x_{2j})$, $j=1,\dots,g$. Let us introduce the following function
\begin{equation}\label{szegoD}
D(z)=\exp\left(\sqrt{\mathcal{R}(z)}\sum_{j=0}^g \int_{-x_{2j+1}}^{-x_{2j}}\frac{(\tilde{\alpha}_0+\cdots+\tilde{\alpha}_j)ds}{\sqrt{\mathcal{R}(s)}(s-z)}\right),
\end{equation}
where $\tilde{\alpha}_0:=\frac{\alpha}{2}$ and the constants $\tilde{\alpha}_j$ are defined by \eqref{def of alpha tilde}. The definition of the $\tilde{\alpha}_{j}$'s ensure that $D(z) \to 1$ as $z \to \infty$. It can be verified that $D$ satisfies the following properties:
\begin{enumerate}
    \item[(a)] $D:\mathbb{C}\setminus(-\infty,0]\to\mathbb{C}$ is analytic.
    
\item[(b)] $D$ has the following jumps:
\begin{align}
& D_+(z)D_-(z)=1, & & z \in (-x_{2j},-x_{2j-1}), ~ j=1,\dots,g+1, \label{jump1 for D} \\
& D_+(z)=D_-(z)e^{2\pi i(\tilde{\alpha}_0+\cdots+\tilde{\alpha}_j)}, & & z \in (-x_{2j+1},-x_{2j}), ~ j=0,\dots,g. \label{jump2 for D}
\end{align}
\item[(c)] As $z\to\infty$, we have
\begin{align}\label{Dinf}
D(z)=\text{exp}\left(\frac{d_1}{\sqrt{z}}+\frac{d_2}{z^\frac{3}{2}}+\bigO(z^{-\frac{5}{2}})\right)=1+\frac{d_1}{\sqrt{z}}+\frac{d_1^2}{2z}+\frac{\frac{d_1^3}{6}+d_2}{z^\frac{3}{2}}+\bigO(z^{-2}),
\end{align}
where $d_{1}$ is given by \eqref{def of d1}.
\item[(d)] As $z\to 0$, we have
\begin{align}
D(z)=z^{\tilde{\alpha}_0}\bigO(1).
\end{align}
\item[(e)] As $z\to-x_j$, $\im z > 0$,
\begin{align}
D(z)=e^{i\pi(\tilde{\alpha}_{0}+\cdots+\tilde{\alpha}_{\tilde{j}})}+\bigO(\sqrt{z+x_j}),
\end{align}
for $j=1,\dots,2g+1$, where $\tilde{j}=\left\lfloor\frac{j}{2}\right\rfloor$.
\end{enumerate}
We are now ready to solve the RH problem for $P^{(\infty)}(z)$.

\begin{proposition}
The solution to the RH problem for $P^{(\infty)}$ is given by
\begin{align}\label{Pinf}
P^{(\infty)}(z)=\begin{pmatrix} 1 & 0 \\ -id_1 & 1 \end{pmatrix}\tilde{P}^{(\infty)}(z)D(z)^{\sigma_3},
\end{align}
where $\nu_{1},\ldots,\nu_{g}$ are defined in \eqref{intro: nu}.
\end{proposition}

\begin{proof}
Using the properties of $D(z)$, $\tilde{P}^{(\infty)}(z)$, and the definition of $\nu_1,\ldots,\nu_g$, we verify that the right-hand side of \eqref{Pinf} solves the RH problem for $P^{(\infty)}$.
\end{proof}

\subsection{Various asymptotics of $P^{(\infty)}(z)$}\label{subsec: Pinf expasions}
For future references, we compute $P^{(\infty)}_{1,12}$ as well as the first terms in the expansions of $P^{(\infty)}(z)$ as $z\to-x_j$, $\im z > 0$, $j=1,\ldots,2g+1$.

\subsubsection{Computation of $P^{(\infty)}_{1,12}$}
From \eqref{tildePinf}, \eqref{tildePinf112}, \eqref{Dinf} we obtain
\begin{align}
    P^{(\infty)}(z)=\left(I+\frac{1}{z}\left[\begin{pmatrix} 1 & 0 \\ -id_1 & 1 \end{pmatrix}\tilde{P}^{(\infty)}_1\begin{pmatrix} 1 & 0 \\ id_1 & 1 \end{pmatrix}+\begin{pmatrix} \frac{d_1^2}{2} & -id_1 \\ -\frac{i}{3}(d_{1}^{3}-3d_{2}) & -\frac{d_1^2}{2} \end{pmatrix}\right]+\bigO(z^{-2})\right)z^{-\frac{\sigma_3}{4}}M \nonumber
\end{align}
as $z\to\infty$, and in particular, using \eqref{tildePinf112}
\begin{align}\label{Pinf112}
P^{(\infty)}_{1,12}=\tilde{P}^{(\infty)}_{1,12} - id_1 =\frac{-2i(\mathbb{A}^{-1})_{g:}\nabla\mathcal{G}(\vec{0})}{\mathcal{G}(\vec{0})} - id_1.
\end{align}

\subsubsection{Asymptotics of $P^{(\infty)}(z)$ as $z\to-x_{2j}$, $\im z>0$, $j=1,\ldots,g$}

Associated to the Jacobian variety $J(X)$ are a total of $2^{2g}$ half-periods, $2^{g-1}(2^g -1)$ of which are odd half-periods, see e.g. \cite[p. 303]{FK1992}. It is well-known, see e.g. \cite[p. 304]{FK1992}, that the $\theta$-function vanishes at each odd half-period. Proposition \ref{AbelMapProperties} allows us to express some of the odd half-periods in terms of the Abel map. 

\begin{lemma}\label{halfperiodlemma}
The points $\vec\varphi(-x_{2j+1}) \in J(X)$, $j = 1, \dots, g$, are odd half-periods. The points $\vec\varphi(-x_{2j}) + \tfrac{\vec{e}_1}{2} \in J(X)$, $j = 1, \dots, g$, are also odd half-periods. 
\end{lemma}
\begin{proof}
It follows from \eqref{AbelMapEndPts} that $\vec{\varphi}(-x_{k})$ can be written in the form
\begin{align*}
\vec{\varphi}(-x_{k}) = \frac{1}{2}\vec{\alpha}_{k} + \frac{1}{2} \tau \vec{\beta}_{k} \mod L(X) \qquad \mbox{ for certain } \vec{\alpha}_{k},\vec{\beta}_{k} \in \mathbb{Z}^{g}, \qquad k=1,\ldots,2g+1.
\end{align*}
Since $\vec{\alpha}_{k}^{t}\vec{\beta}_{k}$ is odd if $k = 2j + 1$ for $j = 1, \dots, g$, we conclude that $\vec\varphi(-x_{2j+1}) \in J(X)$, $j = 1, \dots, g$, are odd half-periods.
It follows similarly from \eqref{AbelMapEndPts} that $\vec\varphi(-x_{2j}) + \tfrac{\vec{e}_1}{2}$, $j = 1, \dots, g$, are odd half-periods. 
\end{proof}

Lemma \ref{halfperiodlemma} implies that 
\begin{align}\label{zeros of theta Abel odd}
\theta\big(\pm\vec{\varphi}_{+}^{\hspace{0.4mm} \mathbb{C}}(-x_{2j})+\tfrac{\vec{e}_1}{2}\big)
= \theta\big(\vec\varphi(-x_{2j})+\tfrac{\vec{e}_1}{2}\big) = 0, \qquad j=1,\dots,g.
\end{align}
Using \eqref{zeros of theta Abel odd}, we infer that the asymptotics of $P^{(\infty)}(z)$ as $z\to-x_{2j}$, $\im z>0$, $j=1,\ldots,g$ are given by
\begin{align}
& P^{(\infty)}(z) = (P^{(\infty)})_{-x_{2j}}^{(-\frac{1}{4})}(z+x_{2j})^{-\frac{1}{4}} + \bigO((z+x_{2j})^{\frac{1}{4}}), \label{Pinf at -x2j} \\
& (P^{(\infty)})_{-x_{2j}}^{(-\frac{1}{4})}=\begin{pmatrix} 1 & 0 \\ -id_1 & 1 \end{pmatrix}\tilde{Q}\frac{(\beta_{2j}^{(\frac{1}{4})})^{\sigma_3}}{\sqrt{2}}\begin{pmatrix} \tilde{\mathcal{G}}_{2j}^- & i\tilde{\mathcal{G}}_{2j}^+ \\ i\mathcal{G}(-\vec{\varphi}_{+}^{\hspace{0.4mm} \mathbb{C}}(-x_{2j})) & \mathcal{G}(\vec{\varphi}_{+}^{\hspace{0.4mm} \mathbb{C}}(-x_{2j})) \end{pmatrix}e^{i\pi(\tilde{\alpha}_j+\cdots+\tilde{\alpha}_j)\sigma_3}, \label{Pinfx2j}
\end{align}
where
\begin{align}
\beta(z)&=\beta_{2j}^{(\frac{1}{4})}(z+x_{2j})^{\frac{1}{4}}+\bigO((z+x_{2j-1})^\frac{5}{4}), \qquad \beta_{2j}^{(\frac{1}{4})}=e^{-\frac{\pi i}{4}}\frac{\prod_{\substack{k=1 \\ k\neq j}}^g|x_{2k}-x_{2j}|^\frac{1}{4}}{\prod_{k=0}^{g}|x_{2k+1}-x_{2j}|^\frac{1}{4}}, \label{beta2j} \\
\tilde{\mathcal{G}}_{2j}^\pm&:=\lim_{\substack{z\to-x_{2j} \\\im z > 0}}\sqrt{z+x_{2j}}\frac{\theta(\pm\vec{\varphi}^{\hspace{0.4mm} \mathbb{C}}(z)+\tfrac{\vec{e}_1}{2}+\vec\nu)}{\theta(\pm\vec{\varphi}^{\hspace{0.4mm} \mathbb{C}}(z)+\tfrac{\vec{e}_1}{2})}=\frac{\theta(\pm\vec{\varphi}_{+}^{\hspace{0.4mm} \mathbb{C}}(-x_{2j})+\frac{\vec{e}_1}{2}+\vec\nu)\prod_{\substack{k=1 \\ k\neq2j}}^{2g+1}\sqrt{x_k-x_{2j}}_{+}}{\pm2\sum_{k=1}^{g}p_{k}(-x_{2j}) \boldsymbol{\partial}_{k}\theta(\pm\vec{\varphi}_{+}^{\hspace{0.4mm} \mathbb{C}}(-x_{2j})+\frac{\vec{e}_1}{2})}. \nonumber
\end{align}
\subsubsection{Asymptotics of $P^{(\infty)}(z)$ as $z \to-x_{1}$, $\im z >0$}
The asymptotics for $P^{(\infty)}(z)$ as $z \to-x_{1}$, $\im z >0$, are given by
\begin{align}
& P^{(\infty)}(z) = (P^{(\infty)})_{-x_{1}}^{(-\frac{1}{4})}(z+x_{1})^{-\frac{1}{4}} + \bigO((z+x_{1})^{\frac{1}{4}}), \label{Pinf at -x1} \\
& (P^{(\infty)})_{-x_{1}}^{(-\frac{1}{4})}=\frac{\beta_1^{(-\frac{1}{4})}\mathcal{G}(\frac{\vec{e}_1}{2})}{\sqrt{2}\mathcal{G}(\vec{0})}\begin{pmatrix} 1 & i \\ -i(c_\mathcal{G}+d_1) & c_\mathcal{G}+d_1 \end{pmatrix}e^{\frac{i\pi\alpha}{2}\sigma_3}, \label{Pinfx1}
\end{align}
where
\begin{align*}
\beta(z)=\frac{\beta_{1}^{(-\frac{1}{4})}}{(z+x_{1})^{\frac{1}{4}}}+\bigO((z+x_{1})^\frac{3}{4}), \qquad \beta_{1}^{(-\frac{1}{4})}=\prod_{k=1}^g\frac{|x_{2k}-x_1|^\frac{1}{4}}{|x_{2k+1}-x_1|^\frac{1}{4}}.
\end{align*}
\subsubsection{Asymptotics of $P^{(\infty)}(z)$ as $z \to-x_{2j+1}$, $\im z >0$, $j=1,\dots,g$}
The asymptotics for $P^{(\infty)}(z)$ as $z \to-x_{2j+1}$, $\im z >0$, $j=1,\dots,g$, are given by
\begin{align}
& P^{(\infty)}(z) = (P^{(\infty)})_{-x_{2j+1}}^{(-\frac{1}{4})}(z+x_{2j+1})^{-\frac{1}{4}} + \bigO((z+x_{2j+1})^{\frac{1}{4}}), \label{Pinf at -x2j+1} \\
& (P^{(\infty)})_{-x_{2j+1}}^{(-\frac{1}{4})}=\begin{pmatrix} 1 & 0 \\ -id_1 & 1 \end{pmatrix}\tilde{Q}\frac{(\beta_{2j+1}^{(-\frac{1}{4})})^{\sigma_3}}{\sqrt{2}}\begin{pmatrix} \mathcal{G}(-\vec{\varphi}_{+}^{\hspace{0.4mm} \mathbb{C}}(-x_{2j+1})+\frac{\vec{e}_1}{2}) & i\mathcal{G}(\vec{\varphi}_{+}^{\hspace{0.4mm} \mathbb{C}}(-x_{2j+1})+\frac{\vec{e}_1}{2}) \\ i\tilde{\mathcal{G}}_{2j+1}^- & \tilde{\mathcal{G}}_{2j+1}^+ \end{pmatrix}e^{i\pi(\tilde{\alpha}_0+\cdots+\tilde{\alpha}_j)\sigma_3}, \nonumber
\end{align}
where
\begin{align}
\beta(z)&=\frac{\beta_{2j+1}^{(-\frac{1}{4})}}{(z+x_{2j+1})^{\frac{1}{4}}}+\bigO((z+x_{2j+1})^\frac{3}{4}), \qquad \beta_{2j+1}^{(-\frac{1}{4})}=\frac{\prod_{k=1}^g|x_{2k}-x_{2j+1}|^\frac{1}{4}}{\prod_{\substack{k=0 \\ k\neq j}}^{g}|x_{2k+1}-x_{2j+1}|^\frac{1}{4}},\label{beta2j-1} \\   
\tilde{\mathcal{G}}_{2j+1}^\pm&:=\lim_{\substack{z\to-x_{2j+1} \\ \im z > 0}}\sqrt{z+x_{2j+1}}\frac{\theta(\pm\vec{\varphi}^{\hspace{0.4mm} \mathbb{C}}(z)+\vec\nu)}{\theta(\pm\vec{\varphi}^{\hspace{0.4mm} \mathbb{C}}(z))}=\frac{\theta(\pm\vec{\varphi}_{+}^{\hspace{0.4mm} \mathbb{C}}(-x_{2j+1})+\vec\nu)\prod_{\substack{k=1 \\ k\neq2j+1}}^{2g+1}\sqrt{x_k-x_{2j+1}}_{+}}{\pm 2 \sum_{k=1}^{g}p_{k}(-x_{2j+1})\boldsymbol{\partial}_{k}\theta(\pm\vec{\varphi}_{+}^{\hspace{0.4mm} \mathbb{C}}(-x_{2j+1}))}.
\end{align}

\section{Local parametrices}\label{section: local parametrix}
The local parametrix $P^{(-x_{j})}$ is defined in a small open disk $\mathbb{D}_{-x_{j}}$ centered at $-x_{j}$, $j=1,\ldots,2g+1$. The radii of the disks are chosen sufficiently small, but independent of $r$, such that they do not intersect each other. Inside $\mathbb{D}_{-x_{j}}$, $P^{(-x_{j})}$ has the same jumps as $S$, and is such that $S(z)P^{(-x_{j})}(z)^{-1}=\bigO(1)$ as $z \to -x_{j}$. Furthermore, on the boundary of $\mathbb{D}_{-x_{j}}$, we impose that $P^{(-x_{j})}$ ``mathches" with $P^{(\infty)}$, in the sense that
\begin{align}\label{matching weak}
P^{(x_{j})}(z) = (I+o(1))P^{(\infty)}(z) \qquad \mbox{as } r \to + \infty,
\end{align}
uniformly for $z \in \partial \mathbb{D}_{-x_{j}}$. 
We will follow the pioneering work \cite{DIZ} and build these parametrices in terms of Bessel functions. We will show in Section \ref{section:smallnorm} that $P^{(-x_{j})}(z)$ approximates $S(z)$ for $z \in \mathbb{D}_{-x_{j}}$.
\subsection{Parametrix at $-x_{2j+1},$ $j=0,\dots,g$}
Let $f_{2j+1}$ be the conformal map from $\mathbb{D}_{-x_{2j+1}}$ to a neighborhood of $0$ defined by
\begin{equation}\label{-x2j-1coord}
f_{2j+1}(z)=\frac{1}{4}\bigg(\gg(z)\mp\frac{i\widehat{\Omega}_{j}}{2}\bigg)^2, \qquad \pm \im z > 0.
\end{equation}
The expansion of $f_{2j+1}(z)$ as $z \to -x_{2j+1}$ is given by
\begin{align}
& f_{2j+1}(z) = c_{2j+1}(z+x_{2j+1})\big( 1+ \bigO(z+x_{2j+1}) \big), & & c_{2j+1}=\frac{q(-x_{2j+1})^2}{\prod_{\substack{k=1 \\ k\neq 2j+1}}^{2g+1}(x_k-x_{2j+1})}>0. \label{c2j-1} 
\end{align}
We deform the lenses such that $f_{2j+1}(\gamma_{\pm}\cap \mathbb{D}_{-x_{2j+1}}) \subseteq e^{\pm \frac{2\pi i}{3}}(0,+\infty)$, and define
\begin{align}
& P^{(-x_{2j+1})}(z):=E_{2j+1}(z)\Phi_{\text{Be}}\left(rf_{2j+1}(z);0\right)e^{-\sqrt{r}\gg(z)\sigma_3}e^{\pm\frac{i\pi\alpha}{2}\sigma_3}, \qquad \pm \im z > 0, \label{P-x2j-1} \\
& E_{2j+1}(z) = P^{(\infty)}(z)e^{\mp(i\pi\alpha-i\widehat{\Omega}_{j}\sqrt{r})\frac{\sigma_3}{2}}M^{-1}\left(2\pi\sqrt{r}f_{2j+1}(z)^{\frac{1}{2}}\right)^{\frac{\sigma_3}{2}}, \qquad \pm \im z > 0, \label{def of E2j-1}
\end{align}
where $\Phi_{\mathrm{Be}}$ is the solution to the RH problem presented in Section \ref{Section:Appendix}. Using the jumps of $P^{(\infty)}$, we verify that $E_{2j+1}(z)$ is analytic for $z\in\mathbb{D}_{-x_{2j+1}}$, and using the jumps \eqref{jump1 g 2cuts}--\eqref{jump3 g 2cuts} of $\mathfrak{g}$, that $P^{(-x_{2j+1})}$ has the same jumps as $S$ inside $\mathbb{D}_{-x_{2j+1}}$ and satisfies $S(z)P^{(-x_{j})}(z)^{-1}=\bigO(1)$ as $z \to -x_{2j+1}$. A direct computation using \eqref{large z asymptotics Bessel} shows that
\begin{multline}\label{P-x2j-1asymp}
P^{(-x_{2j+1})}(z)P^{(\infty)}(z)^{-1}=I + \frac{P^{(\infty)}(z)e^{\mp(i\pi\alpha-i\widehat{\Omega}_{j}\sqrt{r})\frac{\sigma_3}{2}}\Phi_{\mathrm{Be,1}}e^{\pm(i\pi\alpha-i\widehat{\Omega}_{j}\sqrt{r})\frac{\sigma_3}{2}}P^{(\infty)}(z)^{-1}}{\sqrt{r}f_{2j+1}(z)^{\frac{1}{2}}} + \bigO(r^{-1})
\end{multline}
as $r\to +\infty$ uniformly for $z\in\partial\mathbb{D}_{-x_{2j+1}}$, where the $\pm$ and $\mp$ signs correspond to $\pm \im z > 0$.

\subsection{Parametrix at $-x_{2j},$ $j=1,\dots,g$}
Let $f_{2j}$ be the conformal map from $\mathbb{D}_{-x_{2j}}$ to a neighborhood of $0$ defined by
\begin{equation}\label{-x2jcoord}
f_{2j}(z)=-\frac{1}{4}\bigg(\gg(z)\mp\frac{i\widehat{\Omega}_{j}}{2}\bigg)^2, \qquad \pm \im z > 0.
\end{equation}
The expansion of $f_{2j}(z)$ as $z \to -x_{2j}$ is given by
\begin{align}
& f_{2j}(z) = c_{2j}(z+x_{2j})\big(1+ \bigO(z+x_{2j}) \big), & & c_{2j}=\frac{-q^2(-x_{2j})}{\prod_{\substack{k=1 \\ k\neq 2j}}^{2g+1}(x_k-x_{2j})}>0. \label{c2j}
\end{align}
We choose the lenses such that $-f_{2j}(\gamma_{\pm}\cap \mathbb{D}_{-x_{2j}}) \subseteq e^{\pm \frac{2\pi i}{3}}(0,+\infty)$. Let us define $P^{(-x_{2j})}$ by
\begin{align}
& P^{(-x_{2j})}(z):=E_{2j}(z)\Phi_{\text{Be}}(-rf_{2j}(z))\sigma_3e^{-\sqrt{r}\gg(z)\sigma_3}e^{\pm\frac{i\pi\alpha}{2}\sigma_3}, & & \pm \im z > 0, \label{P-x2j} \\
& E_{2j}(z) = P^{(\infty)}(z)e^{\mp(i\pi\alpha-i\widehat{\Omega}_{j}\sqrt{r})\frac{\sigma_3}{2}}\sigma_3M^{-1}\left(2\pi\sqrt{r}(-f_{2j}(z))^{\frac{1}{2}}\right)^{\frac{\sigma_3}{2}}, & & \pm \im z > 0, \label{def of E2j}
\end{align}
where we verify that $E_{2j}(z)$ is analytic for $z\in\mathbb{D}_{-x_{2j}}$. Also, a computation using \eqref{large z asymptotics Bessel} shows that
\begin{multline}\label{P-x2jasymp}
P^{(-x_{2j})}(z)P^{(\infty)}(z)^{-1}=I 
 + \frac{P^{(\infty)}(z)e^{\mp(i\pi\alpha-i\widehat{\Omega}_{j}\sqrt{r})\frac{\sigma_3}{2}}\sigma_3\Phi_{\mathrm{Be,1}}\sigma_3e^{\pm(i\pi\alpha-i\widehat{\Omega}_{j}\sqrt{r})\frac{\sigma_3}{2}}P^{(\infty)}(z)^{-1}}{\sqrt{r}(-f_{2j}(z))^{\frac{1}{2}}} + \bigO(r^{-1})
\end{multline}
as $r\to +\infty$ uniformly for $z\in\partial\mathbb{D}_{-x_{2j}}$, where the $\pm$ and $\mp$ signs correspond to $\pm \im z > 0$.

\section{Small norm problem}\label{section:smallnorm}

Consider
\begin{equation}\label{errorMatrix}
R(z):=\begin{cases}
S(z)P^{(-x_j)}(z)^{-1}, &z\in\mathbb{D}_{-x_j}, ~ j=1,\dots,2g+1, \\
S(z)P^{(\infty)}(z)^{-1}, &z\in\mathbb{C}\setminus\bigcup_{j=1}^{2g+1}\mathbb{D}_{-x_j}.
    \end{cases}
\end{equation}
We have verified in Section \ref{section: local parametrix} that $P^{(-x_j)}$ and $S$ have the same jumps inside $\mathbb{D}_{-x_j}$, and that $R(z)=S(z)P^{(-x_j)}(z)^{-1}$ remains bounded as $z \to -x_{j}$, $j=1,\dots,2g+1$. Thus, using the jumps of $P^{(\infty)}$, we infer that $R$ is analytic in $\mathbb{C}\setminus \Sigma_{R}$, where $\Sigma_{R}$ is given by
\begin{align*}
\Sigma_{R} = \bigg(\gamma_+\cup\gamma_-\cup \bigcup_{j=1}^{2g+1} \partial \mathbb{D}_{-x_{j}} \bigg)  \setminus \bigcup_{j=1}^{2g+1}  \mathbb{D}_{-x_{j}}.
\end{align*}
For convenience, we orient the circles $\partial \mathbb{D}_{-x_{j}}$ in the clockwise direction, see Figure \ref{fig:SigmaR}. The jump matrix for $R$, namely
\begin{align*}
J_{R} : \Sigma_{R} \to \mathbb{C}^{2 \times 2}, \quad z \mapsto J_{R}(z):= R_{-}(z)^{-1}R_{+}(z),
\end{align*}
is given by
\begin{align}\label{Ejump}
    J_{R}(z)=\begin{cases}
        P^{(\infty)}(z)\begin{pmatrix} 1 & 0 \\ e^{\pm i\pi\alpha-2\sqrt{r}\gg(z)} & 1 \end{pmatrix}P^{(\infty)}(z)^{-1}, &z\in\gamma_\pm\setminus\bigcup_{j=1}^{2g+1}\mathbb{D}_{-x_j} \\
        P^{(-x_j)}(z)P^{(\infty)}(z)^{-1}, &z\in\partial\mathbb{D}_{-x_j}, ~ j=1,\dots,2g+1.
    \end{cases}
\end{align}
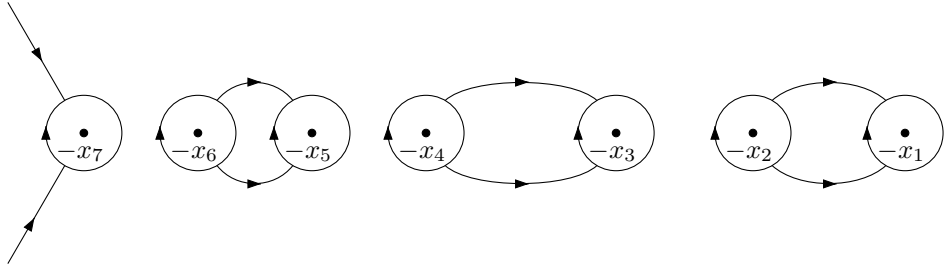
\begin{figure}
\centering
\begin{tikzpicture}

\draw[fill] (0,0) circle (0.05);
\draw[fill] (1.5,0) circle (0.05);
\draw[fill] (3,0) circle (0.05);
\draw[fill] (4.5,0) circle (0.05);
\draw[fill] (7,0) circle (0.05);
\draw[fill] (8.8,0) circle (0.05);
\draw[fill] (10.8,0) circle (0.05);
\draw (-0.25,0.43) -- (120:2.0);
\draw (-0.25,-0.43) -- (-120:2.0);

\draw (0,0) circle (0.5);
\draw (1.5,0) circle (0.5);
\draw (3,0) circle (0.5);
\draw (4.5,0) circle (0.5);
\draw (7,0) circle (0.5);
\draw (8.8,0) circle (0.5);
\draw (10.8,0) circle (0.5);

\draw (1.75,0.43) .. controls (2,0.75) and (2.5,0.75) .. (2.75,0.43);
\draw (1.75,-0.43) .. controls (2,-0.75) and (2.5,-0.75) .. (2.75,-0.43);

\draw (4.75,0.43) .. controls (5.1,0.75) and (6.4,0.75) .. (6.75,0.43);
\draw (4.75,-0.43) .. controls (5.1,-0.75) and (6.4,-0.75) .. (6.75,-0.43);

\draw (9.05,0.43) .. controls (9.4,0.75) and (10.2,0.75) .. (10.55,0.43);
\draw (9.05,-0.43) .. controls (9.4,-0.75) and (10.2,-0.75) .. (10.55,-0.43);

\node at (-0.05,-0.25) {$-x_7$};
\node at (1.45,-0.25) {$-x_6$};
\node at (2.95,-0.25) {$-x_5$};
\node at (4.45,-0.25) {$-x_4$};
\node at (6.95,-0.25) {$-x_3$};
\node at (8.75,-0.25) {$-x_2$};
\node at (10.75,-0.25) {$-x_1$};

\draw[black,arrows={-Triangle[length=0.18cm,width=0.12cm]}]
(-120:1.3) --  ++(60:0.001);
\draw[black,arrows={-Triangle[length=0.18cm,width=0.12cm]}]
(120:1.1) --  ++(-60:0.001);
\draw[black,arrows={-Triangle[length=0.18cm,width=0.12cm]}]
(2.33,0.675) --  ++(0:0.001);
\draw[black,arrows={-Triangle[length=0.18cm,width=0.12cm]}]
(2.33,-0.675) --  ++(0:0.001);
\draw[black,arrows={-Triangle[length=0.18cm,width=0.12cm]}]
(5.85,0.675) --  ++(0:0.001);
\draw[black,arrows={-Triangle[length=0.18cm,width=0.12cm]}]
(5.85,-0.675) --  ++(0:0.001);
\draw[black,arrows={-Triangle[length=0.18cm,width=0.12cm]}]
(9.9,0.675) --  ++(0:0.001);
\draw[black,arrows={-Triangle[length=0.18cm,width=0.12cm]}]
(9.9,-0.675) --  ++(0:0.001);

\draw[black,arrows={-Triangle[length=0.18cm,width=0.12cm]}]
(-0.5,0.09) --  ++(90:0.001);
\draw[black,arrows={-Triangle[length=0.18cm,width=0.12cm]}]
(1,0.09) --  ++(90:0.001);
\draw[black,arrows={-Triangle[length=0.18cm,width=0.12cm]}]
(2.5,0.09) --  ++(90:0.001);
\draw[black,arrows={-Triangle[length=0.18cm,width=0.12cm]}]
(4.0,0.09) --  ++(90:0.001);
\draw[black,arrows={-Triangle[length=0.18cm,width=0.12cm]}]
(6.5,0.09) --  ++(90:0.001);
\draw[black,arrows={-Triangle[length=0.18cm,width=0.12cm]}]
(8.3,0.09) --  ++(90:0.001);
\draw[black,arrows={-Triangle[length=0.18cm,width=0.12cm]}]
(10.3,0.09) --  ++(90:0.001);

\end{tikzpicture}
\caption{The contour $\Sigma_{R}$ for $R$ with $g=3$.}
\label{fig:SigmaR}
\end{figure}
Since $\mathcal{G}(\vec{u}) \neq 0$ for all $\vec{u} \in \mathbb{R}^{g}$, it follows from \eqref{tildePinf} and \eqref{Pinf} that $P^{(\infty)}(z)$ is bounded as $r \to + \infty$ uniformly for $z$ bounded away from $0,-x_{1},\ldots,-x_{2g+1}$. Hence, using Lemma \ref{g_lemma} and \eqref{P-x2j-1asymp}, \eqref{P-x2jasymp}, we obtain
\begin{equation}\label{JR expansion}
J_{R}(z) = \begin{cases}
I + \bigO(e^{-\tilde{c} |rz|^{\frac{1}{2}}}) & \mbox{as } r \to + \infty \mbox{ uniformly for } z \in \Sigma_{R}\setminus\left( \cup_{j=1}^{2g+1} \partial \mathbb{D}_{-x_{j}}\right), \\
I + \frac{J_{R}^{(1)}(z)}{\sqrt{r}} + \bigO(r^{-1}) & \mbox{as } r \to + \infty \mbox{ uniformly for } z \in \cup_{j=1}^{2g+1} \partial \mathbb{D}_{-x_{j}},
\end{cases}
\end{equation}
where $\tilde{c}>0$ is a sufficiently small constant, and $J_{R}^{(1)}(z) = \bigO(1)$ as $r \to + \infty$ uniformly for $z \in \cup_{j=1}^{2g+1} \partial \mathbb{D}_{-x_{j}}$. The matrix $J_{R}^{(1)}(z)$ has been computed in \eqref{P-x2j-1asymp}, \eqref{P-x2jasymp}, and is given by
\begin{align}
J_{R}^{(1)}(z) = \begin{cases}
\frac{P^{(\infty)}(z)e^{\mp(i\pi\alpha-i\widehat{\Omega}_{j}\sqrt{r})\frac{\sigma_3}{2}}\Phi_{\mathrm{Be,1}}e^{\pm(i\pi\alpha-i\widehat{\Omega}_{j}\sqrt{r})\frac{\sigma_3}{2}}P^{(\infty)}(z)^{-1}}{f_{2j+1}(z)^{\frac{1}{2}}}, & z  \in \partial \mathbb{D}_{-x_{2j+1}}, j=0,\dots,g, \\[0.2cm]
\frac{P^{(\infty)}(z)e^{\mp(i\pi\alpha-i\widehat{\Omega}_{j}\sqrt{r})\frac{\sigma_3}{2}}\sigma_3\Phi_{\mathrm{Be,1}}\sigma_3e^{\pm(i\pi\alpha-i\widehat{\Omega}_{j}\sqrt{r})\frac{\sigma_3}{2}}P^{(\infty)}(z)^{-1}}{(-f_{2j}(z))^{\frac{1}{2}}}, & z  \in \partial \mathbb{D}_{-x_{2j}}, j=1,\dots,g.
\end{cases} \label{Jrp1p}
\end{align}
By the small-norm theory for RH problems \cite{DKMVZ2,DeiftZhou}, it follows that $R$ exists for sufficiently large $r$ and
\begin{equation}\label{large r asymptotics for R}
R(z)=I+\frac{R^{(1)}(z)}{\sqrt{r}}+\bigO(r^{-1}) \qquad \mbox{as } r \to + \infty,
\end{equation}
uniformly for $z\in\mathbb{C}\setminus\Sigma_{R}$.  Using \eqref{JR expansion} and the relation
\begin{equation}\label{integral formula for R}
R(z)=I+\frac{1}{2\pi i}\int_{\Sigma_{R}}\frac{R_-(\xi)\left(J_{R}(\xi)-I\right)}{\xi-z}~d\xi, \qquad z \in \mathbb{C}\setminus \Sigma_{R},
\end{equation}
we infer  that
\begin{equation} \label{expression for Rp1p as integral}
R^{(1)}(z)=\frac{1}{2\pi i}\int_{\bigcup_{j=1}^{2g+1}\partial\mathbb{D}_{-x_j}}\frac{J_R^{(1)}(\xi)}{\xi-z}~d\xi, \qquad z \in \mathbb{C}\setminus \bigcup_{j=1}^{2g+1}\partial\mathbb{D}_{-x_j}.
\end{equation}
It is easy to see from \eqref{Jrp1p} that $J_{R}^{(1)}$ admits an analytic continuation from $\cup_{j=1}^{2g+1}\partial\mathbb{D}_{-x_j}$ to $\cup_{j=1}^{2g+1}\overline{\mathbb{D}}_{-x_j}\setminus \{-x_{j}\}$, and furthermore 
\begin{align*}
J_{R}^{(1)}(z) &= (J_{R}^{(1)})_{-x_{j}}^{(-1)} (z+x_j)^{-1} + \bigO(1) \qquad \mbox{as } z \to -x_{j}, \; j=1,\dots,2g+1.
\end{align*}
Recalling that $\partial \mathbb{D}_{-x_{j}}$ is oriented in the clockwise direction, a simple residue calculation shows that, for $z$ outside the disks, \eqref{expression for Rp1p as integral} can be rewritten as
\begin{align}
& R^{(1)}(z)=\sum_{j=1}^{2g+1}\frac{1}{z+x_j} (J_{R}^{(1)})_{-x_j}^{(-1)}, \qquad \mbox{for } z\in\mathbb{C}\setminus\bigcup_{j=1}^{2g+1}\mathbb{D}_{-x_j},
\end{align}
and in particular, we have
\begin{align}\label{R(1)}
    R^{(1)}(z)=\frac{1}{z}\sum_{j=1}^{2g+1}(J_{R}^{(1)})_{-x_j}^{(-1)}+\bigO(z^{-2}) \qquad \mbox{as } z \to \infty.
\end{align} 
Let $R_{1} = \lim_{z \to \infty}z(R(z)-I)$. From \eqref{R(1)}, we infer that
\begin{align}\label{R1 expansion}
R_1(r)=\frac{R_1^{(\frac{1}{2})}}{\sqrt{r}}+\bigO(r^{-1}), \qquad R_1^{(\frac{1}{2})}=\sum_{j=1}^{2g+1}(J_{R}^{(1)})_{-x_j}^{(-1)}.
\end{align}
The matrices $(J_{R}^{(1)})_{-x_j}^{(-1)}$ can be easily computed using \eqref{Pinf at -x2j}, \eqref{Pinf at -x1}, \eqref{Pinf at -x2j+1}, \eqref{c2j-1}, \eqref{c2j}, \eqref{Jrp1p} and the fact that $\det P^{(\infty)}(z) \equiv 1$.  Thus, we have
\begin{align*}
& (J_{R}^{(1)})_{-x_{2j+1}}^{(-1)}=c_{2j+1}^{-\frac{1}{2}}(P^{(\infty)})_{-x_{2j+1}}^{(-\frac{1}{4})}e^{-(i\pi\alpha-i\widehat{\Omega}_{j}\sqrt{r})\frac{\sigma_3}{2}}\Phi_{\mathrm{Be,1}}e^{(i\pi\alpha-i\widehat{\Omega}_{j}\sqrt{r})\frac{\sigma_3}{2}}(P^{(\infty)}_{\text{inv}})_{-x_{2j+1}}^{(-\frac{1}{4})}, & & j=0,\ldots,g, \nonumber\\
& (J_{R}^{(1)})_{-x_{2j}}^{(-1)}=i(c_{2j})^{-\frac{1}{2}}(P^{(\infty)})_{-x_{2j}}^{(-\frac{1}{4})}e^{-(i\pi\alpha-i\widehat{\Omega}_{j}\sqrt{r})\frac{\sigma_3}{2}}\sigma_3\Phi_{\mathrm{Be,1}}\sigma_3e^{(i\pi\alpha-i\widehat{\Omega}_{j}\sqrt{r})\frac{\sigma_3}{2}}(P^{(\infty)}_{\text{inv}})_{-x_{2j}}^{(-\frac{1}{4})}, & & j=1,\dots,g, 
\end{align*}
where
\begin{align*}
(P^{(\infty)}_{\text{inv}})_{-x_{j}}^{(-\frac{1}{4})} = \begin{pmatrix}
(P^{(\infty)})_{-x_{j},22}^{(-\frac{1}{4})} & -(P^{(\infty)})_{-x_{j},12}^{(-\frac{1}{4})} \\
-(P^{(\infty)})_{-x_{j},21}^{(-\frac{1}{4})} & (P^{(\infty)})_{-x_{j},11}^{(-\frac{1}{4})}
\end{pmatrix}, \qquad j=1, \ldots,2g+1.
\end{align*}
In particular, using \eqref{intro: ThetaPeriodic}, we find
\begin{align}
\left[(J_{R}^{(1)})_{-x_{2j+1}}^{(-1)}\right]_{12}&=\frac{(\beta_{2j+1}^{(-\frac{1}{4})})^2\mathcal{G}(\vec{\varphi}_{+}^{\hspace{0.4mm} \mathbb{C}}(-x_{2j+1})+\tfrac{\vec{e}_1}{2})\mathcal{G}(-\vec{\varphi}_{+}^{\hspace{0.4mm} \mathbb{C}}(-x_{2j+1})+\tfrac{\vec{e}_1}{2})}{16i\sqrt{c_{2j+1}}\mathcal{G}(\vec{0})^2}, & & j=0,\ldots,g, \label{lol7} \\
\left[(J_{R}^{(1)})_{-x_{2j}}^{(-1)}\right]_{12}&=\frac{(\beta_{2j}^{(\frac{1}{4})})^2\tilde{\mathcal{G}}_{2j}^+\tilde{\mathcal{G}}_{2j}^-}{16\sqrt{c_{2j}}\mathcal{G}(\vec{0})^2}, & & j=1,\ldots,g. \label{lol8}
\end{align}
More precisely, the right-hand side of \eqref{lol7} for $j=1,\ldots,g$ has been obtained after some simplifications using $\theta(\vec{\varphi}_{+}^{\hspace{0.4mm} \mathbb{C}}(-x_{2j+1})+\frac{e_{1}}{2}+\vec{\nu}) = e^{-2\pi i (\nu_{1}+\cdots+\nu_{j})} \theta(-\vec{\varphi}_{+}^{\hspace{0.4mm} \mathbb{C}}(-x_{2j+1})+\frac{e_{1}}{2}+\vec{\nu})$, and the right-hand side of \eqref{lol8} has been obtained using the relation $\tilde{\mathcal{G}}_{2j}^{+} = -e^{-2\pi i (\nu_{1}+\cdots+\nu_{j})}\tilde{\mathcal{G}}_{2j}^{-}$.

We note from \eqref{lol7} and \eqref{lol8} that the expressions for $\big[(J_{R}^{(1)})_{-x_{j}}^{(-1)}\big]_{12}$ appear unrelated for odd and even $j$. In Proposition \ref{PropR112} below, we show that they can all (both for $j$ even and odd) be rewritten in an unified way.

\section{Proofs of Theorems \ref{thm: main result - DIZ analogue}-\ref{thm: main result - ergodic}}\label{section: proof of thms}
Recall from Proposition \ref{prop: diff id} that
\begin{align}\label{partialrlogF}
\partial_{r} \log F(r\vec{x}) = \tfrac{1}{2ir}\Phi_{1,12}(r) + \tfrac{1-4\alpha^{2}}{16r}. 
\end{align}

\subsection{Proof of Theorem \ref{thm: main result - DIZ analogue}}
The proof relies on equation (\ref{partialrlogF}). We first compute the large $r$ asymptotics of $\Phi_{1,12}(r)$. 

\begin{proposition}\label{PropPhi12Initial}
We have
\begin{align*}
\frac{\Phi_{1,12}(r)}{2ir}=c+\frac{P^{(\infty)}_{1,12}+R_{1,12}}{2i\sqrt{r}}=c+\frac{P^{(\infty)}_{1,12}}{2i\sqrt{r}}+\frac{R^{(\frac{1}{2})}_{1,12}}{2ir}+\bigO(r^{-\frac{3}{2}}) \qquad \mbox{as } r\to+\infty.
\end{align*}
\end{proposition}
\begin{proof}
From \eqref{def of T}, \eqref{def:S} and \eqref{errorMatrix}, we have, for all sufficiently large $z$,
\begin{align}
\Phi(z)=r^{-\frac{\sigma_3}{4}}\begin{pmatrix} 1 & 0 \\ 2ic\sqrt{r} & 1 \end{pmatrix}R(z)P^{(\infty)}(z)e^{\sqrt{r}\gg(z)\sigma_3}. \nonumber
\end{align}
Taking $z\to\infty$
and using \eqref{def of Phi1}, \eqref{asymp g 2cuts}, \eqref{eq:Pinf asympinf} and \eqref{R1 expansion}, we obtain the claim.
\end{proof}

Our next goal is to rewrite $\frac{P^{(\infty)}_{1,12}}{2i\sqrt{r}}$ as an $r$-derivative. To prepare ourselves for that matter, we first prove the following lemma.

\begin{lemma}\label{pOmegalemma}
We have $(\mathbb{A}^{-1})_{gj} = \frac{\Omega_j}{4\pi}$, $j=1,\ldots,g$.
\end{lemma}
\begin{proof}
Let us define
\begin{align}\label{wj differentials}
w_r := \frac{z^{r-1}}{\sqrt{\mathcal{R}(z)}}dz, \quad r=1,\ldots,g+1, \qquad \mbox{so that} \qquad \omega_r = \sum_{k=1}^g w_k (\mathbb{A}^{-1})_{kr}, \qquad r = 1, \dots, g.
\end{align}
Using the definition \eqref{intro: Omegaj} of $\Omega_{j}$, and noting that $a_{r,\ell} = \int_{A_r} w_{\ell}$, we get
\begin{align}
\frac{\Omega_j}{4\pi} = \frac{1}{4\pi i} \int_{B_j}\frac{q(s)}{\sqrt{\mathcal{R}(s)}}ds  & = \frac{1}{8\pi i} \int_{B_j} w_{g+1} - \frac{1}{8\pi i}  \sum_{k=1}^g\sum_{r=1}^g \bigg(\int_{B_j} w_k \bigg) (\mathbb{A}^{-1})_{kr} a_{r,g+1} \nonumber \\
& = \frac{1}{8\pi i}\bigg( \int_{B_j} w_{g+1} - \sum_{r=1}^g \int_{B_j} \omega_r \int_{A_r} w_{g+1} \bigg). \label{Omegadifference1}
\end{align}
The lemma will follow from (\ref{Omegadifference1}) if we can show that
\begin{align}\label{intBjAr}
\int_{B_j} w_{g+1} - \sum_{r=1}^g \int_{A_r} w_{g+1} \int_{B_j} \omega_r 
= 8\pi i (\mathbb{A}^{-1})_{gj}.
\end{align}
To prove (\ref{intBjAr}), first note that the differential $w_{g+1}$ is holomorphic except at $z = \infty$. As $z \to \infty$, we have $\sqrt{\mathcal{R}(z)} = \sqrt{z}(z^{g} + O(z^{g-1}))$, and hence $w_{g+1} = \frac{dz}{\sqrt{z}}(1+ O(z^{-1}))$. In terms of the local analytic coordinate $u := z^{-\frac{1}{2}}$ at $z = \infty$, this becomes
\begin{align*}
w_{g+1} = -2\frac{du}{u^2} + O(1) \qquad \mbox{as } u \to 0.
\end{align*}
It follows that the differential $\tilde{\eta}$, defined by
\begin{align}\label{tildethetadef}
\tilde{\eta} = -\frac{1}{2}\bigg[w_{g+1} - \sum_{r=1}^g \bigg(\int_{A_r} w_{g+1}\bigg) \omega_r\bigg],
\end{align}
is an Abelian differential of the second kind with vanishing $A$-periods whose only pole is a double pole at $z = \infty$ (i.e. at $u = 0$); moreover, at this pole $\tilde{\eta}$ has the singular part $du/u^2$. 
Applying Riemann's bilinear identity to $\tilde{\eta}$ and a holomorphic differential $\eta$, we find (see e.g. \cite[p. 64, eq. (3.0.2)]{FK1992})
\begin{align}\label{bilinearidentity}
\sum_{l=1}^g \bigg(\int_{A_l} \eta \int_{B_l} \tilde{\eta}- \int_{B_l} \eta \int_{A_l} \tilde{\eta}\bigg)
= 2\pi i \underset{u = 0}{\res} (f\tilde{\eta}),
\end{align}
where $f$ is a function such that $df =\eta$ near $u = 0$. 
Let us fix $j \in \{1, \dots, g\}$ and choose $\eta = \omega_j$. Then (\ref{bilinearidentity}) reduces to
\begin{align}\label{intBjtildetheta}
\int_{B_j} \tilde{\eta} = 2\pi i \underset{u = 0}{\res} (f_j\tilde{\eta}),
\end{align}
where $df_j = \omega_j$ near $u = 0$. Using the definition \eqref{omegaJ} of $\omega_{j}$, we obtain
\begin{align*}
\omega_j =z^{-3/2} (\mathbb{A}^{-1})_{gj} (1+ O(z^{-1}))dz 
= -2 (\mathbb{A}^{-1})_{gj} (1+ O(u^2))du \qquad \mbox{as } u \to 0,
\end{align*}
which implies $f_j(u) = f_j(0) - 2 (\mathbb{A}^{-1})_{gj} u+ O(u^2)$ as $u \to 0$,
and thus
\begin{align}\label{resftildetheta}
2\pi i \underset{u = 0}{\res} (f_j\tilde{\eta}) = - 4\pi i (\mathbb{A}^{-1})_{gj}.
\end{align}
The identity (\ref{intBjAr}) follows by substituting (\ref{tildethetadef}) and (\ref{resftildetheta}) into (\ref{intBjtildetheta}). This completes the proof.
\end{proof}

\begin{proposition}\label{prop: P112 derivative}
We have the identity
\begin{align*}
\frac{P^{(\infty)}_{1,12}}{2i\sqrt{r}} = - \frac{d_1}{2\sqrt{r}} -\frac{1}{\sqrt{r}} \sum_{j=1}^{g} (\mathbb{A}^{-1})_{gj}\frac{\boldsymbol{\partial}_{j}\mathcal{G}(\vec{0})}{\mathcal{G}(\vec{0})} =\frac{d}{dr}\left[\log\theta(\vec\nu)-d_1\sqrt{r}\right].
\end{align*}
\end{proposition}

\begin{proof}
The first equality follows from \eqref{Pinf112}. Note that $\mathcal{G}(\vec{0})$ depends on $r$, whereas $d_1$ and $(\mathbb{A}^{-1})_{gj}$ do not.
 Let $\vec{z}=(z_{1},\ldots,z_{g})\in\mathbb{C}^g$. Using the definition \eqref{mathcalGdef} of $\mathcal{G}$, we get
\begin{align*}
\frac{\boldsymbol{\partial}_{j}\mathcal{G}(\vec{z})}{\mathcal{G}(\vec{z})}=\frac{\boldsymbol{\partial}_{j}\theta(\vec{z}+\vec\nu)}{\theta(\vec{z}+\vec\nu)}-\frac{\boldsymbol{\partial}_{j}\theta(\vec{z})}{\theta(\vec{z})}, \qquad j=1,\ldots,g.
\end{align*}
Since $\theta(\vec{z})=\theta(-\vec{z})$ for all $\vec{z}\in\mathbb{C}^g$, we have $\boldsymbol{\partial}_{j}\theta(\vec{0})=0$ for all $j=1,\dots,g$, and thus
\begin{align*}
\frac{\boldsymbol{\partial}_{j}\mathcal{G}(\vec{0})}{\mathcal{G}(\vec{0})} = \frac{\boldsymbol{\partial}_{j}\theta(\vec\nu)}{\theta(\vec\nu)}, \qquad j=1,\ldots,g.
\end{align*}
Lemma \ref{pOmegalemma} gives the relation $\frac{d\nu_j}{dr}=-\frac{(\mathbb{A}^{-1})_{gj}}{\sqrt{r}}$, so we conclude that
\begin{align*}
-\frac{1}{\sqrt{r}} \sum_{j=1}^{g} (\mathbb{A}^{-1})_{gj}\frac{\boldsymbol{\partial}_{j}\mathcal{G}(\vec{0})}{\mathcal{G}(\vec{0})} &= \sum_{j=1}^{g} \frac{d  \nu_{j}}{dr} \frac{\boldsymbol{\partial}_{j}\theta(\vec\nu)}{\theta(\vec\nu)} = \frac{d}{dr}\log\theta(\vec\nu).
\end{align*}
\end{proof}

Finding a simplified expression for $R_{1,12}^{(\frac{1}{2})}$ is more challenging and requires some preparation. A divisor on $X$ is a formal product of points of the form
\begin{align*}
P_{1}^{\alpha_{1}}\cdots P_{k}^{\alpha_{k}}, \qquad \mbox{ where } \quad P_{1},\ldots,P_{k} \in X, \quad \alpha_{1},\ldots,\alpha_{k} \in \mathbb{Z}.
\end{align*}
Let $X_g$ denote the $n$-fold symmetric product of the Riemann surface $X$, that is, $X_{g}$ is the set of all divisors of the form $P_{1} \cdots P_{g}$ where $P_{1},\ldots,P_{g} \in X$. It is known \cite[section III.11.9]{FK1992} that $X_{g}$ can be given a complex structure, so that $X_{g}$ is a manifold of complex dimension $g$. The Abel map $\vec{\varphi}$ was defined in \eqref{def of Abel map} as a map from $X$ to $J(X)$. Following \cite[p.  \hspace{-1mm}92]{FK1992}, we extend the map $\vec{\varphi}$ to $X_{g}$ as follows:
\begin{align}\label{def of Abel map over Xg}
\vec{\varphi}: X_{g} \to J(X), \qquad D=P_{1} \cdots P_{g} \mapsto \vec{\varphi}(D)= \sum_{j=1}^{g}\int_{-x_1}^{P_{j}} \vec{\omega}^{t}.
\end{align}

Let $\uppi:X\to\mathbb{C}$ denote the projection of $X$ onto $\mathbb{C}$, and let $\mathcal{A}$ be the subset of $X_{g}$ consisting of all divisors $D=P_1\cdots P_g$ such that $P_j\in \uppi^{-1}([-x_{2j+1},-x_{2j}])$, $j=1,\dots,g$. One can equip $\mathcal{A}$ with an atlas of smooth local charts inherited from the atlas of analytic local charts of $X_{g}$, so that $\mathcal{A}$ is a smooth submanifold of $X_g$ of real dimension $g$. 

We recall from \cite[p. 110 (a)]{FK1992} that a divisor $P_1\cdots P_g$ is special if and only if there exists a non-zero holomorphic differential $\tilde{\omega}$ which vanishes at each of the points $P_{1},\ldots,P_{g}$. 

\begin{lemma}\label{lemma: nonspecial}
Each divisor in $\mathcal{A}$ is non-special.
\end{lemma}
\begin{proof}
Since $\omega_{1},\ldots,\omega_{g}$ form a basis of the space of holomorphic differentials, it follows that any holomorphic differential can be written in the form $\tilde{\omega}=\tilde{p}(z)dz/\sqrt{\mathcal{R}(z)}$ where $\tilde{p}$ is a polynomial of degree at most $g-1$. If $P_1\cdots P_g \in \mathcal{A}$, then by definition of $\mathcal{A}$ we have $P_{j_{1}} \neq P_{j_{2}}$ for $1 \leq j_{1} \neq j_{2} \leq g$. Note that $dz/\sqrt{\mathcal{R}(z)}$ vanishes only at $\infty$. Hence, $\tilde{\omega}$ vanishes at $P_{j}$ if and only if $\tilde{p}$ vanishes at $P_{j}$. Since there exists no non-zero polynomial of degree $\leq g-1$ vanishing at $g$ distinct points, the proof is complete. 
\end{proof}

According to Lemma \ref{halfperiodlemma}, the $g$ points $\vec{\varphi}(-x_{3}), \vec{\varphi}(-x_{5}), \dots, \vec{\varphi}(-x_{2g+1})$ are odd half-periods in $J(X)$. Following \cite[p. 325, eq. (1.2.1)]{FK1992}, we define the vector of Riemann constants $\mathcal{K} \in J(X)$ by
\begin{align}
\mathcal{K} & := \vec{\varphi}((-x_{3})(-x_{5}) \cdots (-x_{2g+1})) \nonumber \\
& = \sum_{j=1}^g\vec{\varphi}(-x_{2j+1})=\frac{g}{2}\vec{e}_1-\frac{1}{2}\sum_{j=2}^g\vec{e}_j+\frac{1}{2}\sum_{j=1}^g(g-j+1)\vec{\tau}_j \mod L(X), \label{RiemannConstant}
\end{align}
where we have used \eqref{AbelMapEndPts} for the last equality. 

We recall the following result from \cite[p. 312]{FK1992}.

\begin{lemma}\label{Dlemma}
If $e = \vec{\varphi}(D) + \mathcal{K}$ and $D \in X_g$, then the multi-valued function $X \ni P \mapsto \theta(\vec{\varphi}(P) - e)$ vanishes identically if and only if $D$ is special. Furthermore, this function has a well-defined zero set, and if $D$ is not special, then $D$ is the divisor of zeros of $X \ni P \mapsto \theta(\vec{\varphi}(P) - e)$.
\end{lemma}

Clearly, the set 
\begin{align*}
\mathcal{T} & := \{\mathcal{K} + \vec{x} + \tau\vec{m}: \vec{x} \in \mathbb{R}^{g}, \; \vec{m}\in\mathbb{Z}^g\}/L(X) = \mathcal{K} + \mathbb{R}^{g}/\mathbb{Z}^{g}
\end{align*}
is a nonempty connected smooth submanifold of $J(X)$ of real dimension $g$. Let $\vec{\varphi}|_\mathcal{A}$ denote the restriction of $\vec{\varphi}:X_g \to J(X)$ to $\mathcal{A}$. 

\begin{proposition}\label{varphiAprop}
The map $\vec{\varphi}|_\mathcal{A}: \mathcal{A} \to \mathcal{T}$ is a diffeomorphism, that is, $\vec{\varphi}|_\mathcal{A}: \mathcal{A} \to \mathcal{T}$ is a smooth bijection, and $(\vec{\varphi}|_{\mathcal{A}})^{-1}:\mathcal{T} \to \mathcal{A}$ is smooth.
\end{proposition} 
\begin{proof}
Let $D=P_1\cdots P_g\in \mathcal{A}$. Since $\mathcal{K} = \vec{\varphi}((-x_{3})(-x_{5}) \cdots (-x_{2g+1}))$, we have
\begin{align*}
\vec{\varphi}(D)=\mathcal{K}+ \sum_{j=1}^g \int_{-x_{2j+1}}^{P_j}\vec{\omega} \; \in \; \mathcal{T},
\end{align*}
where we have used that the restriction of $\vec{\omega}$ to $\uppi^{-1}([-x_{2j+1},-x_{2j}])$, $j=1,\ldots,g$, takes values in $\mathbb{R}^g$. This shows that $\vec{\varphi}(\mathcal{A}) \subseteq \mathcal{T}$. 

It readily follows from its definition that the map $\vec{\varphi}|_\mathcal{A}: \mathcal{A} \to \mathcal{T}$ is smooth. We have shown in Lemma \ref{lemma: nonspecial} that all divisors in $\mathcal{A}$ are non-special. Let $D=P_{1}\cdots P_{g}$ be an arbitrary divisor in $\mathcal{A}$. Since $D$ is non-special, the tangent map of $\vec{\varphi}:X_g \to J(X)$ at $D$ has full rank (see e.g. \cite[section VI.2.5]{FK1992}); hence the tangent map of $\vec{\varphi}|_{\mathcal{A}}:\mathcal{A} \to \mathcal{T}$ at $D$ also has full rank. 
Therefore, by the inverse function theorem, $\vec{\varphi}|_\mathcal{A}:\mathcal{A} \to \mathcal{T}$ is a local diffeomorphism.

Since $\vec{\varphi}|_\mathcal{A}$ is continuous, $\mathcal{A}$ is compact, and $\mathcal{T}$ is Hausdorff, it follows that $\vec{\varphi}|_\mathcal{A}$ is a proper map \cite[Proposition A.53]{LeeSM}. 
Therefore, we have shown that $\vec\varphi|_\mathcal{A}$ is a proper local diffeomorphism. Since $\mathcal{A}$ and $\mathcal{T}$ are nonempty and connected, this implies that $\vec\varphi|_\mathcal{A}$ is a smooth covering map \cite[Proposition 4.46]{LeeSM}, that is, $\vec\varphi|_\mathcal{A} : \mathcal{A} \to \mathcal{T}$ is a smooth surjective map, and for each $\vec{t} \in \mathcal{T}$, there exists a neighborhood $U \subseteq \mathcal{T}$ of $\vec{t}$ such that each component of $(\vec\varphi|_\mathcal{A})^{-1}(U)$ is mapped diffeomorphically onto $U$ by $\vec\varphi|_\mathcal{A}$. 



Let $D_{1},D_{2} \in \mathcal{A}$. Lemma \ref{lemma: nonspecial} implies that $D_{1}$ and $D_{2}$ are non-special. Hence, Lemma \ref{Dlemma} implies that $D_{j}$ is the divisor of zeros of $P \mapsto \theta(\vec{\varphi}(P)-\vec{\varphi}(D_{j})-\mathcal{K})$, $j=1,2$. In particular, if $\vec\varphi(D_{1})=\vec\varphi(D_{2})$, then necessarily $D_{1}=D_{2}$. This shows that $\vec{\varphi}|_{\mathcal{A}}$ is injective.
Since an injective smooth covering map is a diffeomorphism, see \cite[Proposition 4.33 (b)]{LeeSM}, the proof is complete.
\end{proof}

Let $\vec{\mu} \in \R^g/ \mathbb{Z}^{g}$.  In view of Proposition \ref{varphiAprop}, there exist $\tilde{D}=\tilde{P}_1\cdots\tilde{P}_g \in \mathcal{A}$ and $\hat{D}=\hat{P}_1\cdots\hat{P}_g \in \mathcal{A}$ such that
\begin{align}\label{zero divisor}
\tilde{D} = (\vec{\varphi}|_\mathcal{A})^{-1}(- \vec{\mu} - \tfrac{\vec{e}_1}{2} - \mathcal{K}), \qquad \hat{D} = (\vec{\varphi}|_\mathcal{A})^{-1}(\vec{\mu} + \tfrac{\vec{e}_1}{2} - \mathcal{K}),
\end{align}
or equivalently
\begin{align}\label{tildeDhatD}
\vec{\mu} + \frac{\vec{e}_1}{2}  = -\vec{\varphi}(\tilde{D}) - \mathcal{K} \mod L(X), \qquad  -\vec{\mu} - \frac{\vec{e}_1}{2}  = -\vec{\varphi}(\hat{D}) - \mathcal{K} \mod L(X).
\end{align}
It follows from Lemma \ref{lemma: nonspecial} that $\tilde{D}$ and $\hat{D}$ are non-special, and from Lemma \ref{Dlemma} that 
\begin{align*}
P \mapsto \theta(\vec{\varphi}(P) + \tfrac{\vec{e}_1}{2} + \vec{\mu}) = \theta(\vec{\varphi}(P) -\vec{\varphi}(\tilde{D})-\mathcal{K}) \quad \mbox{and} \quad P \mapsto \theta(-\vec{\varphi}(P) + \tfrac{\vec{e}_1}{2} + \vec{\mu}) = \theta(\vec{\varphi}(P) -\vec{\varphi}(\hat{D})-\mathcal{K})
\end{align*} 
do not vanish identically and that their zero divisors are given by $\tilde{D}$ and $\hat{D}$, respectively. Let $\iota:X \to X$ be the sheet-changing involution. Since $\vec{\varphi}(P) = -\vec{\varphi}(\iota(P))$ and $\mathcal{K} = -\mathcal{K} \mod L(X)$, it follows from \eqref{tildeDhatD} that $\tilde{P}_k = \iota(\hat{P}_k)$ (and in particular $\uppi(\tilde{P}_k)=\uppi(\hat{P}_k)$) for $k=1,\dots,g$. Let us define
\begin{align}\label{bk}
b_k(\vec{\mu}):=\uppi(\tilde{P}_k)=\uppi(\hat{P}_k)\in[-x_{2k+1},-x_{2k}], ~~~ k=1,\dots,g.
\end{align}

\begin{proposition}\label{ThetaIdentity}
For each $P \in X$ and $\vec{\mu}\in \mathbb{R}^{g}/\mathbb{Z}^{g}$, we have
\begin{align}\label{lol6}
\frac{\theta^2(\vec{0})}{\theta^2(\vec{\mu})}\frac{\theta(\vec{\varphi}(P)+\frac{\vec{e}_1}{2}+\vec\mu)\theta(-\vec{\varphi}(P)+\frac{\vec{e}_1}{2}+\vec\mu)}{\theta(\vec{\varphi}(P)+\frac{\vec{e}_1}{2})\theta(-\vec{\varphi}(P)+\frac{\vec{e}_1}{2})}=\prod_{k=1}^g\frac{z-b_k(\vec\mu)}{z+x_{2k}}, \qquad z=\uppi(P).
\end{align}
\end{proposition}
\begin{proof}
Let $h(P)$ denote the left-hand side of \eqref{lol6}. Using \eqref{intro: ThetaPeriodic}, it is easily verified that the function $\hat{h}$ defined by
\begin{align*}
\vec{u} \mapsto \hat{h}(\vec{u}):= \frac{\theta^2(\vec{0})}{\theta^2(\vec{\mu})}\frac{\theta(\vec{u}+\frac{\vec{e}_1}{2}+\vec\mu)\theta(-\vec{u}+\frac{\vec{e}_1}{2}+\vec\mu)}{\theta(\vec{u}+\frac{\vec{e}_1}{2})\theta(-\vec{u}+\frac{\vec{e}_1}{2})}, \qquad \vec{u} \in \mathbb{C}^{g},
\end{align*}
satisfies
\begin{align*}
\hat{h}(\vec{u}) = \hat{h}(\vec{u}+\vec{\lambda}'+\tau\vec{\lambda}), \qquad  \vec{\lambda},\vec{\lambda}'\in\mathbb{Z}^g.
\end{align*}
This shows that $h(P) = \hat{h}(\varphi(P))$ is a well-defined function of $P \in X$. From the discussion above \eqref{bk}, we conclude that $h$ has exactly $2g$ zeros (counting multiplicities) at $\tilde{P}_{1},\ldots,\tilde{P}_{g},\hat{P}_{1},\ldots,\hat{P}_{g}$. Also, it follows from Lemma \ref{halfperiodlemma} that $\vec{\varphi}(-x_{2k})+\frac{\vec{e}_1}{2}$ are odd half-periods, and therefore the function $\theta(\vec{\varphi}(P)+\frac{\vec{e}_1}{2}) =\theta(-\vec{\varphi}(P)+\frac{\vec{e}_1}{2})$ has $g$ simple zeros at $-x_{2},-x_{4},\ldots,-x_{2g}$. In fact, $\theta(\vec{\varphi}(P)+\frac{\vec{e}_1}{2})$ has no other zeros, because from \eqref{AbelMapEndPts} and \eqref{RiemannConstant}, we deduce that 
\begin{align*}
\vec{\varphi}((-x_{2})(-x_{4}) \cdots (-x_{2g})) = \mathcal{K} + \frac{\vec{e}_{1}}{2},
\end{align*}
and thus Lemma \ref{Dlemma} ensures that $(-x_{2})(-x_{4}) \cdots (-x_{2g})$ is the divisor of zeros of $\theta(\vec{\varphi}(P)+\frac{\vec{e}_1}{2})$.

Since $\vec{\varphi}(P) = -\vec{\varphi}(\iota(P))$, $h$ satisfies $h(P) = h(\iota(P))$ for all $P \in X$. In order words, $h(P)$ only depends on $\uppi(P)$ and therefore $h$ can be viewed as a meromorphic function on the Riemann sphere. It directly follows from the above discussion and from \eqref{bk} that $h$, viewed as a function on the Riemann sphere, has $g$ simple zeros at $b_{1}(\vec{\mu}),\ldots,b_{g}(\vec{\mu})$, and $g$ simple poles at $-x_{2},-x_{4},\ldots,-x_{2g}$. Hence the ratio of the left- and right-hand sides of \eqref{lol6} is equal to a constant $\kappa$. By Proposition \ref{AbelMapProperties}, $\vec{\varphi}(\infty) = \frac{\vec{e}_1}{2} \mod L(X)$. This shows that $\kappa=1$, which finishes the proof.
\end{proof}
Define the function $\mathcal{B}:\mathbb{C}\times \mathbb{R}^{g}/\mathbb{Z}^{g}\to\mathbb{C} \cup \{\infty\}$ by
\begin{align}\label{B}
\mathcal{B}(z,\vec{u}):=\frac{\prod_{k=1}^g(z-b_k(\vec{u}))}{q(z)} =\frac{\theta^2(\vec{0})}{\theta^2(\vec{u})}\frac{\theta(\vec{\varphi}(P)+\frac{\vec{e}_1}{2}+\vec{u})\theta(-\vec{\varphi}(P)+\frac{\vec{e}_1}{2}+\vec{u})}{\theta(\vec{\varphi}(P)+\frac{\vec{e}_1}{2})\theta(-\vec{\varphi}(P)+\frac{\vec{e}_1}{2})}\frac{\prod_{k=1}^g(z+x_{2k})}{q(z)},
\end{align}
where $P\in X$ is such that $z=\uppi(P)$. The second equality in \eqref{B} follows from Proposition \ref{ThetaIdentity}.

\begin{proposition}\label{PropR112}
Recall that $\vec{\nu}(r) \in \mathbb{R}^{g}$ is defined in \eqref{intro: nu}. We have
\begin{align*}
R_{1,12}^{(\frac{1}{2})}&=\frac{1}{16i}\sum_{j=1}^{2g+1}\mathcal{B}(-x_j,\vec{\nu}(r)).
\end{align*}
\end{proposition}

\begin{proof}
After a direct computation using \eqref{beta2j}, \eqref{beta2j-1}, \eqref{c2j-1}, \eqref{c2j}, \eqref{lol7}, \eqref{lol8}, the definition of $\mathcal{G}$ \eqref{mathcalGdef}, and the right-most expression for $\mathcal{B}$ in \eqref{B}, we obtain
\begin{align*}
\frac{1}{16i}\mathcal{B}(-x_{j},\vec{\nu}(r))&=\left[(J_{R}^{(1)})_{-x_j}^{(-1)}\right]_{12}, \qquad j=1,2,\ldots,2g+1.
\end{align*}
By \eqref{R1 expansion}, $R_{1,12}^{(\frac{1}{2})}=\sum_{j=1}^{2g+1}\big[(J_{R}^{(1)})_{-x_j}^{(-1)}\big]_{12}$, and thus the claim follows.
\end{proof}

We are now ready to compute the large $r$ asymptotics of $\frac{\Phi_{1,12}(r)}{2ir}$.

\begin{proposition}\label{prop:Phi1 as a derivative}
As $r\to+\infty$, we have
\begin{align}\label{Phi112 simplified}
\frac{\Phi_{1,12}(r)}{2ir}=\frac{d}{dr}\left[cr-d_1\sqrt{r}+\log\theta(\vec\nu)- \frac{1}{32} \int_{M}^{r}\sum_{j=1}^{2g+1}\mathcal{B}(-x_j,\vec{\nu}(t))\frac{dt}{t}\right] +\bigO(r^{-\frac{3}{2}}),
\end{align}
where $M>0$ is independent of $r$.
\end{proposition}
\begin{proof}
It follows from Propositions \ref{PropPhi12Initial}, \ref{prop: P112 derivative} and \ref{PropR112} that
\begin{align*}
\frac{\Phi_{1,12}(r)}{2ir} & = c+\frac{P^{(\infty)}_{1,12}}{2i\sqrt{r}}+\frac{R^{(\frac{1}{2})}_{1,12}}{2ir} +\bigO(r^{-\frac{3}{2}})  = c+ \frac{d}{dr}\left[\log\theta(\vec\nu)-d_1\sqrt{r}\right] - \frac{1}{32r}\sum_{j=1}^{2g+1}\mathcal{B}(-x_j,\vec{\nu}(r))  +\bigO(r^{-\frac{3}{2}})
\end{align*}
as $r \to + \infty$, which is equivalent to (\ref{Phi112 simplified}).
\end{proof}

We obtain \eqref{asymptotics in main thm} after substituting the asymptotics \eqref{Phi112 simplified} into (\ref{partialrlogF}), integrating in $r$, and then exponentiating both sides. To finish the proof of Theorem \ref{thm: main result - DIZ analogue}, it remains to prove \eqref{leading term is a space average}. 
\begin{proposition}\label{ergodic: extract log r}
Let $M > 0$ and let $\mathcal{H}:\mathbb{R}^{g}/ \mathbb{Z}^{g}\to\mathbb{R}$ be continuous. Then $\widehat{\mathcal{H}} \in \mathbb{R}$ is well-defined by
$$\widehat{\mathcal{H}} = \lim_{\mathrm{T} \to +\infty} \frac{1}{\mathrm{T}} \int_0^\mathrm{T} \mathcal{H}( \vec\nu(t^{2}))dt$$
and, for all $r > 0$, 
\begin{align}
\int_M^r\mathcal{H}(\vec{\nu}(r'))\frac{dr'}{r'} & = \frac{2}{\sqrt{r}} \int_{\sqrt{M}}^{\sqrt{r}}\mathcal{H}(\vec{\nu}(t^{2}))dt + 2 \int_{\sqrt{M}}^{\sqrt{r}} \frac{1}{t}\bigg( \frac{\int_{\sqrt{M}}^{t}\mathcal{H}(\vec{\nu}(t'^{2}))dt'}{t}-\widehat{\mathcal{H}} \bigg)dt + \widehat{\mathcal{H}}\log \frac{r}{M}.
\label{exact expression for integral main thm} 
\end{align}
In particular, 
\begin{align}
\int_M^r\mathcal{H}(\vec{\nu}(r'))\frac{dr'}{r'} = \widehat{\mathcal{H}} \log r +o(\log r) \qquad \mbox{as } r \to + \infty. \label{Birkhoff modified with 1/r}
\end{align}
\end{proposition}
\begin{proof}
Because $\mathbb{R}^{g}/ \mathbb{Z}^{g}$ is compact and $\mathcal{H}$ is continuous, $t \mapsto \mathcal{H}(\vec{\nu}(t^{2}))$ is an almost-periodic function (see \cite[Definition 5.1]{Katznelson}), and thus $\widehat{\mathcal{H}} \in \mathbb{R}$ is well-defined (see \cite[p. 176]{Katznelson}). The identity \eqref{exact expression for integral main thm} can be verified by differentiating both sides. Since the first and second terms on the right-hand side of \eqref{exact expression for integral main thm} are $\bigO(1)$ and $o(\log r)$ as $r \to + \infty$, respectively, \eqref{Birkhoff modified with 1/r} follows.
\end{proof}
Applying Proposition \ref{ergodic: extract log r} to $\mathcal{H}=\mathcal{B}(-x_{j},\cdot)$, we get
\begin{align*}
\int_M^r \frac{\mathcal{B}(-x_j,\vec{\nu}(r'))}{r'}dr' = \mathcal{B}_{j} \log r+o(\log r) \qquad \mbox{as } r \to + \infty,
\end{align*}
which is \eqref{leading term is a space average}. This finishes the proof of Theorem \ref{thm: main result - DIZ analogue}.

\subsection{Proof of Theorem \ref{thm: main result - periodic}}
In this subsection, whose content is inspired by \cite[p. 156]{DIZ}, we assume that there exist $\delta_{1},\delta_{2}>0$ such that (\ref{good diophantine prop}) holds.

\begin{proposition}\label{prop: gperiodic log int}
Let $\mathcal{H}:\mathbb{R}^{g}/ \mathbb{Z}^{g}\to\mathbb{R}$ be analytic and assume that \eqref{good diophantine prop} holds. For any $M>0$, we have
\begin{align}
& \frac{1}{t}\int_{\sqrt{M}}^{t}\mathcal{H}(\vec{\nu}(t'^{2}))dt' = \widehat{\mathcal{H}}+ \bigO(t^{-1}) \qquad \mbox{as } t \to + \infty, \label{lol9}
\end{align}
where $\widehat{\mathcal{H}} = \lim_{\mathrm{T} \to +\infty} \frac{1}{\mathrm{T}} \int_0^\mathrm{T} \mathcal{H}( \vec\nu(t^{2}))dt$.
\end{proposition}
\begin{proof}
We already know from Proposition \ref{ergodic: extract log r} that $\frac{1}{t}\int_{\sqrt{M}}^{t}\mathcal{H}(\vec{\nu}(t'^{2}))dt' = \widehat{\mathcal{H}}+ o(1)$ as $t \to + \infty$, so it only remains to show that the error term can be reduced from $o(1)$ to $\bigO(t^{-1})$. Since $\mathcal{H}$ is analytic on $\mathbb{R}^{g}/ \mathbb{Z}^{g}$, its Fourier series
\begin{align*}
\sum_{\vec{m} \in \mathbb{Z}^{g}} \ell_{\vec{m}}e^{2\pi i \vec{m}^{t}\vec{u}}
\end{align*}
converges to $\mathcal{H}(\vec{u})$ uniformly for $\vec{u} \in \mathbb{R}^{g}/\mathbb{Z}^{g}$, and the Fourier coefficients $\ell_{\vec{m}}$ decay exponentially fast, i.e., there exist $c_{1},c_{2}>0$ such that
\begin{align}\label{exponential decays of fourier coeff}
|\ell_{\vec{m}}| \leq c_{1} e^{-c_{2} \| \vec{m} \|} \qquad \mbox{for all } \vec{m} \in \mathbb{Z}^{g}.
\end{align}
Hence, 
\begin{align*}
\frac{1}{t}\int_{\sqrt{M}}^{t}\mathcal{H}(\vec{\nu}(t'^{2}))dt' & = \frac{1}{t} \int_{\sqrt{M}}^{t} \sum_{\vec{m} \in \mathbb{Z}^{g}} \ell_{\vec{m}}e^{-i \vec{m}^{t}(\vec{\Omega}t+2\pi\vec{\tilde{\alpha}})}dt = \frac{1}{t} \sum_{\vec{m} \in \mathbb{Z}^{g}} \int_{\sqrt{M}}^{t}  \ell_{\vec{m}}e^{- i \vec{m}^{t}(\vec{\Omega}t+2\pi\vec{\tilde{\alpha}})}
	 \\
& = \frac{1}{t} \sum_{\substack{\vec{m} \in \mathbb{Z}^{g} \\ \vec{m}^{t} \vec{\Omega} \neq 0}} \ell_{\vec{m}}e^{-2\pi i \vec{m}^{t}\vec{\tilde{\alpha}}} \frac{e^{- i \vec{m}^{t}\vec{\Omega}t}-e^{-i \vec{m}^{t}\vec{\Omega}\sqrt{M}}}{-i \vec{m}^{t}\vec{\Omega}} + \frac{1}{t}\sum_{\substack{\vec{m} \in \mathbb{Z}^{g} \\ \vec{m}^{t} \vec{\Omega} = 0}} \ell_{\vec{m}}e^{-2\pi i \vec{m}^{t}\vec{\tilde{\alpha}}} (t-\sqrt{M}).
\end{align*}
The claim now follows from \eqref{good diophantine prop} and \eqref{exponential decays of fourier coeff}.
\end{proof}
Applying Proposition \ref{prop: gperiodic log int} to $\mathcal{H}=\mathcal{B}(-x_{j},\cdot)$ gives
\begin{align}
& \frac{1}{t}\int_{\sqrt{M}}^{t}\mathcal{B}(-x_{j},\vec{\nu}(t'^{2}))dt' = \mathcal{B}_{j}+ \bigO(t^{-1}) \qquad \mbox{as } t \to + \infty, \label{lol11}
\end{align}
and substituting these asymptotics into \eqref{exact expression for integral main thm} (again with $\mathcal{H}=\mathcal{B}(-x_{j},\cdot)$), we get
\begin{align}
\int_M^r \mathcal{B}(-x_j,\vec{\nu}(r'))\frac{dr'}{r'} = \mathcal{B}_{j} \log r + C_{j}+ \bigO(r^{-\frac{1}{2}}) \qquad \mbox{as } r \to + \infty,
\end{align}
for some $C_{j}$ independent of $r$. This comples the proof of \eqref{good error term inside main thm dioph}, and hence also of Theorem \ref{thm: main result - periodic}.

\subsection{Proof of Theorem \ref{thm: main result - ergodic}}
In this subsection we assume that $\Omega_{1},\ldots,\Omega_{g}$ are rationally independent. This condition implies that the linear flow
\begin{align*}
(1,+\infty) \ni t \mapsto (\nu_{1}(t^{2}) \hspace{-2.4mm} \mod 1 \; , \, \nu_{2}(t^{2}) \hspace{-2.4mm} \mod 1 \; , \ldots, \, \nu_{g}(t^{2})\hspace{-2.4mm} \mod 1)
\end{align*}
is ergodic in $\mathbb{R}^{g}/ \mathbb{Z}^{g}$, where we recall that $\vec{\nu}(r)$ is given by \eqref{intro: nu}. Given a function $\mathcal{H}: \mathbb{R}^{g}/\mathbb{Z}^{g} \to \mathbb{R}$, the Birkhoff ergodic theorem states that the average of $\mathcal{H}$ along the above flow (often referred to as the \textit{time average} in the literature) is equal to the average of $\mathcal{H}$ over the torus $\mathbb{R}^{g}/ \mathbb{Z}^{g}$ (the \textit{space average}). More precisely, the time average of $\mathcal{H}$ is defined by
\begin{align}\label{time average}
\lim_{\mathrm{T} \to \infty} \frac{1}{\mathrm{T}} \int_M^\mathrm{T} \mathcal{H}( \vec\nu(t^{2}) + \vec{x})dt, \qquad \vec{x} \in \mathbb{R}^{g}, \quad M>0,
\end{align}
(this limit is clearly independent of $M$) and the space average of $\mathcal{H}$, which we denote by $\langle \mathcal{H} \rangle$, is defined by
\begin{align}
\langle \mathcal{H} \rangle = \int_{[0,1]^{g}}\mathcal{H}(\vec{x})d\vec{x} = \int_{0}^{1} \int_{0}^{1} \cdots \int_{0}^{1} \mathcal{H}(x_{1},x_{2},\ldots,x_{g})dx_{1}dx_{2}\cdots dx_{g}.
\end{align} 
Because $\Omega_{1},\ldots,\Omega_{g}$ are rationally independent, Birkhoff's ergodic theorem states that (see e.g. \cite[p. 286]{A1989})
\begin{align}\label{timeequalsspaceaverage}
\lim_{\mathrm{T} \to \infty} \frac{1}{\mathrm{T}} \int_M^\mathrm{T} \mathcal{H}( \vec\nu(t^{2}) + \vec{x})dt = \langle \mathcal{H} \rangle \qquad \mbox{for every } \vec{x} \in \mathbb{R}^{g}.
\end{align}
In particular, the time average is independent of $\vec{x}$. To complete the proof of Theorem \ref{thm: main result - ergodic}, it remains to evaluate $\langle \mathcal{B}(-x_{j},\cdot ) \rangle$ explicitly for $j=1,2,\ldots,2g+1$.

Let us orient the closed loop $\uppi^{-1}([-x_{2j+1},-x_{2j}])$ so that $\uppi^{-1}([-x_{2j+1},-x_{2j}])$ is homotopic to $A_{j}-A_{j+1}$, where $A_{1},\ldots,A_{g}$ are the cycles shown in Figure \ref{fig:homology}. 
This induces an orientation on $\mathcal{A}$ via the identification
$$\mathcal{A} \cong \uppi^{-1}([-x_{3},-x_{2}]) \times \uppi^{-1}([-x_{5},-x_{4}]) \times \cdots \times \uppi^{-1}([-x_{2g+1},-x_{2g}]).$$
Our next lemma shows that with this orientation of $\mathcal{A}$, the map $\vec{\varphi}|_{\mathcal{A}}-\mathcal{K} - \frac{\vec{e}_1}{2}$ is an orientation-preserving diffeomorphism from $\mathcal{A}$ to the subset $\mathbb{R}^{g}/\mathbb{Z}^{g}$ of the Jacobian variety $J(X)$.

\begin{lemma}\label{orientationlemma}
The map $D \mapsto f(D) := \vec{\varphi}(D) - \mathcal{K} - \frac{\vec{e}_1}{2}$ is an orientation-preserving diffeomorphism from $\mathcal{A}$ to $\mathbb{R}^{g}/\mathbb{Z}^{g}$. Moreover, the inverse of $f$ satisfies $\uppi \circ f^{-1} = \vec{b}$, where $\vec{b}:\mathbb{R}^{g}/\mathbb{Z}^{g} \to \mathbb{R}^g$ is defined by
\begin{align}\label{vecbdef}
\vec{b}(\vec{u}) = (b_{1}(\vec{u}),\ldots,b_{g}(\vec{u})), \qquad \vec{u} \in \mathbb{R}^{g}/\mathbb{Z}^{g}
\end{align}
and $\uppi:\mathcal{A} \to \mathbb{R}^g$ is defined by $\uppi(D) = (\uppi(P_1), \dots, \uppi(P_g))$ for $D = P_1 \cdots P_g \in \mathcal{A}$.
\end{lemma}
\begin{proof}
By Proposition \ref{varphiAprop}, the map $\vec{\varphi}|_{\mathcal{A}}-\mathcal{K} - \frac{\vec{e}_1}{2}$ is a diffeomorphism from $\mathcal{A}$ to $\mathbb{R}^{g}/\mathbb{Z}^{g}$. 
To see that it is orientation-preserving, note that the multi-valued function 
\begin{align*}
\vec{\phi} :\mathcal{A}\to \mathbb{C}^{g}, \qquad D=P_{1} \cdots P_{g} \mapsto \vec{\phi}(D)= \sum_{j=1}^{g}\int_{-x_1}^{P_{j}} \vec{\omega}^{t}
\end{align*}
increases by $\vec{f}_{j} = \vec{e}_{j}-\vec{e}_{j+1}$ as $P_{j}$ goes once around $\uppi^{-1}([-x_{2j+1},-x_{2j}])$. Since the matrix relating the two bases $\{\vec{e}_{j}\}_{j=1}^{g}$ and $\{\vec{f}_{j}\}_{j=1}^{g}$ has determinant $+1$, the map $\vec{\varphi}|_{\mathcal{A}}-\mathcal{K}:\mathcal{A} \to \mathbb{R}^{g}/\mathbb{Z}^{g}$ is orientation-preserving.

Suppose $f(D) = \vec{u}$ where $D \in \mathcal{A}$ and $\vec{u} \in \mathbb{R}^{g}/\mathbb{Z}^{g}$. Then $\vec{u} + \frac{\vec{e}_1}{2} = \vec{\varphi}(D) + \mathcal{K}$ in $J(X)$, so by Lemma \ref{lemma: nonspecial} and Lemma \ref{Dlemma}, $D$ is the divisor of zeros of $\theta(\vec{\varphi}(\cdot) - \vec{u} - \frac{\vec{e}_1}{2})$.
The definition (\ref{bk}) of $b_k$ then implies that $\uppi(D) = \vec{b}(\vec{u})$. This shows that $\uppi(f^{-1}(\vec{u})) = \vec{b}(\vec{u})$ as desired.
\end{proof}

\begin{proposition}\label{B avg q relation}
For all $z\in\mathbb{C}$, we have
\begin{align}
\left\langle \prod_{k=1}^g(z-b_k(\cdot))\right\rangle = \det\left(zI-T\right),
\end{align}
where $b_{1},\ldots,b_{g}$ are defined in \eqref{bk}, and the $g\times g$ matrix $T$ is defined by
\begin{align}\label{T matrix}
T=\left(T_{ij}\right)_{i,j=1}^g, \qquad T_{ij}:=\int_{A_i}z\omega_j = \int_{A_{i}} \frac{dz}{\sqrt{\mathcal{R}(z)}} (z \;\; z^{2} \;\; \cdots \; \; z^{g})(\mathbb{A}^{-1}) \vec{e}_{j}.
\end{align}
\end{proposition} 
\begin{proof}
Fix $z_0 \in \mathbb{C}$. Defining the function $F$ by
\begin{align}
F:\uppi(\mathcal{A}) \cong [-x_{3},-x_{2}] \times [-x_{5},-x_{4}] \times \cdots \times [-x_{2g+1},-x_{2g}] \to \mathbb{C}, \qquad F(\vec{\mathrm{b}})=\prod_{k=1}^g(z_0-\mathrm{b}_k),
\end{align} 
and the function $\vec{b}:\mathbb{R}^{g}/\mathbb{Z}^{g} \to \mathbb{R}^g$ by (\ref{vecbdef}),
we can write the space average of $F \circ \vec{b}$ as
\begin{align}\label{spaceaverageFb}
\big\langle F \circ \vec{b} \big\rangle = \int_{[0,1]^{g}} F(\vec{b}(\vec{u}))d\vec{u} 
\end{align}
By Lemma \ref{orientationlemma}, the map $f := \vec{\varphi}|_{\mathcal{A}}-\mathcal{K} - \frac{\vec{e}_1}{2}$ is an orientation-preserving diffeomorphism from $\mathcal{A}$ to $\mathbb{R}^{g}/\mathbb{Z}^{g} \cong [0,1)^{g}$. Hence 
\begin{align}\label{langleFbrangle}
\big\langle F \circ \vec{b} \big\rangle
& = \int_{f(\mathcal{A})} F \circ \vec{b}\, d\vec{u}
= \int_{\mathcal{A}} f^*(F \circ \vec{b}\, d\vec{u}),
\end{align}
where $f^*(F \circ \vec{b}\, d\vec{u})$ denotes the pull-back by $f$ of the $g$-form $F \circ \vec{b}\, d\vec{u}$. (The last step in (\ref{langleFbrangle}) amounts to changing variables from $\vec{u} \in \mathbb{R}^{g}/\mathbb{Z}^{g}$ to $D \in \mathcal{A}$ according to $u_k = \varphi_k(D) -\mathcal{K}_{k} - \frac{(\vec{e}_1)_k}{2}$.)
In the same way that we use the analytic coordinate $z$ to perform calculations on $X$ (viewing $X$ as a two-sheeted cover of the complex $z$-plane), we can use $(z_1, \dots, z_g)$ as coordinates on 
$$\mathcal{A} \cong \uppi^{-1}([-x_{3},-x_{2}]) \times \uppi^{-1}([-x_{5},-x_{4}]) \times \cdots \times \uppi^{-1}([-x_{2g+1},-x_{2g}]) \subset X \times \cdots \times X,$$
where $z_j$ denotes the coordinate on the $j$-th factor of $X$. Evaluating the integral on the right-hand side of (\ref{langleFbrangle}) in these coordinates and using that $f^*(F \circ \vec{b}) = F \circ \vec{b} \circ f = F \circ \uppi$ by Lemma \ref{orientationlemma}, we obtain 
\begin{align*}
\big\langle F \circ \vec{b} \big\rangle & = \prod_{r=1}^g \int_{A_r - A_{r+1}} dz_r  \det(\varphi_i'(z_j))_{i,j=1}^g F(z_1, \dots, z_g)
	\\
& = \prod_{r=1}^g \int_{A_r - A_{r+1}} dz_r  \det \begin{pmatrix}
\varphi_{1}'(z_{1}) & \varphi_{1}'(z_{2}) & \ldots & \varphi_{1}'(z_{g}) \\
\varphi_{2}'(z_{1}) & \varphi_{2}'(z_{2}) & \ldots & \varphi_{2}'(z_{g}) \\
\vdots & \vdots & \ddots & \vdots \\
\varphi_{g}'(z_{1}) & \varphi_{g}'(z_{2}) & \ldots & \varphi_{g}'(z_{g})
\end{pmatrix}  \prod_{k=1}^g (z_0- z_k)
	\\
& = \prod_{r=1}^g \int_{A_r - A_{r+1}} dz_r  \det \begin{pmatrix}
\varphi_1'(z_{1})(z_0-z_{1}) & \varphi_1'(z_{2})(z_0-z_{2}) & \ldots & \varphi_1'(z_{g})(z_0-z_{g}) \\
\varphi_2'(z_{1})(z_0-z_{1}) & \varphi_2'(z_{2})(z_0-z_{2}) & \ldots & \varphi_2'(z_{g})(z_0-z_{g}) \\
\vdots & \vdots & \ddots & \vdots \\
\varphi_g'(z_{1})(z_0-z_{1}) & \varphi_g'(z_{2})(z_0-z_{2}) & \ldots & \varphi_g'(z_{g})(z_0-z_{g})
\end{pmatrix}.
\end{align*}
Using that $\varphi_j'(z) dz = \omega_j$ and that $\oint_{A_k}\omega_j =\delta_{jk}$ for $j,k = 1, \ldots,g$, we conclude that 
\begin{align*}
\langle F \circ \vec{b} \rangle 
= \det
 \begin{pmatrix}
z_0-\int_{A_1 - A_{2}} z\omega_1 & -\int_{A_2 - A_{3}} z\omega_1 & \ldots & -\int_{A_{g-1}-A_{g}} z\omega_1 & -\int_{A_g} z\omega_1 \\
-z_0-\int_{A_1 - A_{2}} z\omega_2 & z_0 -\int_{A_2 - A_{3}} z\omega_2 & \ldots & -\int_{A_{g-1}-A_{g}} z\omega_2 & -\int_{A_g} z\omega_2 \\
-\int_{A_1 - A_{2}} z\omega_3 & -z_0 -\int_{A_2 - A_{3}} z\omega_3 & \ldots & -\int_{A_{g-1}-A_{g}} z\omega_3 & -\int_{A_g} z\omega_3 \\
\vdots & \vdots & \ddots & \vdots & \vdots \\
-\int_{A_1 - A_{2}} z\omega_{g-1} & -\int_{A_2 - A_{3}} z\omega_{g-1} & \ldots & z_0-\int_{A_g-A_{g-1}} z\omega_{g-1} & -\int_{A_g} z\omega_{g-1} \\
-\int_{A_1 - A_{2}} z\omega_{g} & -\int_{A_2 - A_{3}} z\omega_{g} & \ldots & -z_0-\int_{A_g-A_{g-1}} z\omega_{g} & z_0-\int_{A_g} z\omega_{g}
\end{pmatrix}.
\end{align*}
By first adding the $g$-th column to the $(g-1)$-th column, then the $(g-1)$-th column to the $(g-2)$-th column etc., we obtain
\begin{align*}
\langle F \circ \vec{b} \rangle 
& =\det \begin{pmatrix}
z_0-\int_{A_1 - A_{2}} z\omega_1 & -\int_{A_2 - A_{3}} z\omega_1 & \ldots & -\int_{A_{g-1}} z\omega_1 & -\int_{A_g} z\omega_1 \\
-z_0-\int_{A_1 - A_{2}} z\omega_2 & z_0 -\int_{A_2 - A_{3}} z\omega_2 & \ldots & -\int_{A_{g-1}} z\omega_2 & -\int_{A_g} z\omega_2 \\
-\int_{A_1 - A_{2}} z\omega_3 & -z_0 -\int_{A_2 - A_{3}} z\omega_3 & \ldots & -\int_{A_{g-1}} z\omega_3 & -\int_{A_g} z\omega_3 \\
\vdots & \vdots & \ddots & \vdots & \vdots \\
-\int_{A_1 - A_{2}} z\omega_{g-1} & -\int_{A_2 - A_{3}} z \omega_{g-1} & \ldots & z_0-\int_{A_{g-1}} z\omega_{g-1} & -\int_{A_g} z\omega_{g-1} \\
-\int_{A_1 - A_{2}} z\omega_g & -\int_{A_2 - A_{3}} z\omega_g & \ldots & -\int_{A_{g-1}} z\omega_g & z_0-\int_{A_g} z\omega_g
\end{pmatrix}  
	\\
& = \cdots = \det \begin{pmatrix}
z_0-\int_{A_1} z\omega_1 & -\int_{A_2} z\omega_1 & \ldots & -\int_{A_{g-1}} z\omega_1 & -\int_{A_g} z\omega_1 \\
-\int_{A_1} z\omega_2 & z_0 -\int_{A_2} z\omega_2 & \ldots & -\int_{A_{g-1}} z\omega_2 & -\int_{A_g} z\omega_2 \\
-\int_{A_1} z\omega_3 & -\int_{A_2} z\omega_3 & \ldots & -\int_{A_{g-1}} z\omega_3 & -\int_{A_g} z\omega_3 \\
\vdots & \vdots & \ddots & \vdots & \vdots \\
-\int_{A_1} z\omega_{g-1} & -\int_{A_2} z\omega_{g-1} & \ldots & z_0-\int_{A_{g-1}} z\omega_{g-1} & -\int_{A_g} z\omega_{g-1} \\
-\int_{A_1} z\omega_g & -\int_{A_2} z\omega_g & \ldots & -\int_{A_{g-1}} z\omega_g & z_0-\int_{A_g} z\omega_g
\end{pmatrix}   \\
& = \det(z_0I-T).
\end{align*}
Since $z_0 \in \mathbb{C}$ was arbitrary, this finishes the proof.
\end{proof}

\begin{proposition}\label{qT relation}
For all $z \in \mathbb{C}$, we have the identity
\begin{align}
q(z)=\frac{1}{2}\det\left(zI-T\right),
\end{align}
where $q(z)$ is defined in \eqref{intro: q}.
\end{proposition}
\begin{proof}
Let $\hat{\mathbb{A}}$ be the $g \times g$ matrix defined by
\begin{align}\label{hatA}
\hat{\mathbb{A}}:=(a_{i,j+1})_{i,j=1}^g, \qquad  a_{i,j+1}=\int_{A_i}w_{j+1} = \int_{A_i}\frac{z^{j}dz}{\sqrt{R(z)}},
\end{align}
where we recall that $w_{j+1}$ is defined in \eqref{wj differentials}. We can write the definition \eqref{T matrix} of $T_{ij}$ as
\begin{align*}
T_{ij} = \int_{A_i}z\omega_j=\sum_{k=1}^g\hat{\mathbb{A}}_{ik}(\mathbb{A}^{-1})_{kj}, \qquad i,j=1,\ldots,g,
\end{align*}
and hence
\begin{align}
T=\hat{\mathbb{A}}\mathbb{A}^{-1}.
\end{align}
On the other hand, since the $j$-th column of $\hat{\mathbb{A}}$ equals the $(j+1)$-th column of $\mathbb{A}$ for $j=1,\dots,g-1$, the matrix $\mathbb{A}^{-1}\hat{\mathbb{A}}$ has the form
\begin{align}
    \mathbb{A}^{-1}\hat{\mathbb{A}}=\begin{pmatrix} 0 & 0 & \cdots & 0 & * \\ 1 & 0 & \cdots & 0 & * \\ 0 & 1 & \cdots & 0 & * \\ \vdots & \vdots & \ddots & \vdots & \vdots \\ 0 & 0 & \cdots & 1 & * \end{pmatrix},
\end{align}
and by \eqref{intro: q} the $g$-th column of $\mathbb{A}^{-1}\hat{\mathbb{A}}$ is given by
\begin{align*}
\mathbb{A}^{-1}\hat{\mathbb{A}}\vec{e}_g&=\mathbb{A}^{-1} \vec{a}=-2(q_{0},\ldots,q_{g-1})^{t}.
\end{align*}
Hence, we find
\begin{align*}
\frac{1}{2}\det\left(zI-T\right)=\frac{1}{2}\det\left(zI-\hat{\mathbb{A}}\mathbb{A}^{-1}\right)=\frac{1}{2}\det\left(zI-\mathbb{A}^{-1}\hat{\mathbb{A}}\right)=\frac{z^{g}}{2} - \frac{1}{2}\sum_{j=1}^{g}(\mathbb{A}^{-1}\hat{\mathbb{A}})_{jg}z^{j-1}=q(z),
\end{align*}
where we have expanded the determinant along the last column to obtain the third equality.
\end{proof}

We conclude from (\ref{Birkhoff modified with 1/r}), (\ref{timeequalsspaceaverage}), and the definition (\ref{calBdef}) of $\mathcal{B}$ that
\begin{align*}
\int_M^r\frac{\mathcal{B}(-x_j,\vec{\nu}(t))}{t}dt
& = \langle \mathcal{B}(-x_j, \cdot)\rangle \log r +o(\log r)
	\\
& = \bigg\langle \frac{\prod_{k=1}^g(-x_j-b_k(\cdot))}{q(-x_j)} \bigg\rangle \log r +o(\log r), \qquad r \to + \infty, ~ j=1,\ldots,2g+1.
\end{align*}
Propositions \ref{B avg q relation} and \ref{qT relation} imply that 
$$\bigg\langle \prod_{k=1}^g(-x_j-b_k(\cdot)) \bigg\rangle
= 2q(-x_j), \qquad j=1,\ldots,2g+1.$$
Consequently, as $r \to + \infty$,
$$\int_M^r\frac{\mathcal{B}(-x_j,\vec{\nu}(t))}{t}dt
= 2\log r +o(\log r), \qquad  j=1,\ldots,2g+1,$$
which proves (\ref{explicit leading term}) and completes the proof of Theorem \ref{thm: main result - ergodic}.

\appendix
\section{On the mapping $(x_{1},x_{2},\ldots,x_{2g+1}) \to (\Omega_{1},\ldots,\Omega_{g})$}\label{appendix: mapping open ball}
Since $q$ has exactly one zero in each of the intervals $(-x_{2g+1},-x_{2g}), \dots, (-x_{5},-x_{4})$, $(-x_{3},-x_{2})$, we have
\begin{align*}
q(z) \to \frac{1}{2}\prod_{j=1}^{g}(z+x_{2j}) \qquad \mbox{as } (x_{2},x_{4},\ldots,x_{2g}) \to (x_{3},x_{5},\ldots,x_{2g+1})
\end{align*}
uniformly for $z$ in compact subsets of $\mathbb{C}$. Hence, by \eqref{intro: Omegaj}, for each $j \in \{1,2,\ldots,g\}$,
\begin{align*}
\Omega_{j} \to 2 \big( \sqrt{x_{2j+1}-x_{1}} - \sqrt{x_{2j-1}-x_{1}} \big) \qquad \mbox{as } (x_{2},x_{4},\ldots,x_{2g}) \to (x_{3},x_{5},\ldots,x_{2g+1}),
\end{align*}
and thus for each $\vec{x}_{*}:=(x_{1},x_{3},x_{3},x_{5},x_{5},\cdots,x_{2g+1},x_{2g+1})$ satisfying $0<x_{1}<x_{3}<x_{5}<\ldots<x_{2g+1}$, $\Omega(\vec{x}_{*})$ is analytic in each of the variables $x_{1},x_{3},x_{5},\ldots,x_{2g+1}$ and we have
\begin{align*}
\det \Big( \frac{d (\Omega_{j}(\vec{x}_{*}))}{d x_{2k+1}} \Big)_{j,k=1}^{g} = \det \Big( (\boldsymbol{\partial}_{2k}\Omega_{j})(\vec{x}_{*})) + (\boldsymbol{\partial}_{2k+1}\Omega_{j})(\vec{x}_{*})) \Big)_{j,k=1}^{g} = \prod_{j=1}^{g} \frac{1}{\sqrt{x_{2j+1}-x_{1}}} \neq 0.
\end{align*}
This shows that the image of the mapping 
\begin{align*}
\{(x_{1},x_{2},\ldots,x_{2g+1}): x_{1}<x_{2}<\cdots<x_{2g+1}\} \to (\Omega_{1},\ldots,\Omega_{g})
\end{align*}
contains an open ball in $(0,+\infty)^{g}$, so that for $g \geq 2$ each of the four cases in \eqref{the four cases} does happen for certain choices of $(x_{1},x_{2},\ldots,x_{2g+1})$.

\section{Bessel model RH problem}\label{Section:Appendix}
\begin{itemize}
\item[(a)] $\Phi_{\mathrm{Be}} : \mathbb{C} \setminus \Sigma_{\mathrm{Be}} \to \mathbb{C}^{2\times 2}$ is analytic, where
$\Sigma_{\mathrm{Be}}=(-\infty,0) \cup e^{ \frac{2\pi i}{3} }  (0,+\infty) \cup e^{ -\frac{2\pi i}{3} }  (0,+\infty)$ is oriented as shown in Figure \ref{figBessel}.
\item[(b)] $\Phi_{\mathrm{Be}}$ satisfies the jump conditions
\begin{equation}\label{Jump for P_Be}
\begin{array}{l l} 
\Phi_{\mathrm{Be},+}(z) = \Phi_{\mathrm{Be},-}(z) \begin{pmatrix}
0 & 1 \\ -1 & 0
\end{pmatrix}, & z \in (-\infty,0), \\

\Phi_{\mathrm{Be},+}(z) = \Phi_{\mathrm{Be},-}(z) \begin{pmatrix}
1 & 0 \\ 1 & 1
\end{pmatrix}, & z \in e^{ \frac{2\pi i}{3} }  (0,+\infty), \\

\Phi_{\mathrm{Be},+}(z) = \Phi_{\mathrm{Be},-}(z) \begin{pmatrix}
1 & 0 \\ 1 & 1
\end{pmatrix}, & z \in e^{ -\frac{2\pi i}{3} }  (0,+\infty). \\
\end{array}
\end{equation}
\item[(c)] As $z \to \infty$, $z \notin \Sigma_{\mathrm{Be}}$, we have
\begin{equation}\label{large z asymptotics Bessel}
\Phi_{\mathrm{Be}}(z) = ( 2\pi z^{\frac{1}{2}} )^{-\frac{\sigma_{3}}{2}}M
\left(I+\frac{ \Phi_{\mathrm{Be},1}}{z^{\frac{1}{2}}} + \bigO(z^{-1})\right) e^{2 z^{\frac{1}{2}}\sigma_{3}},
\end{equation}
where $\ds \Phi_{\mathrm{Be},1} = \frac{1}{16}\begin{pmatrix}
-1 & -2i \\ -2i & 1
\end{pmatrix}$ and $M = \frac{1}{\sqrt{2}}\begin{pmatrix}
1 & i \\ i & 1
\end{pmatrix}$.
\item[(d)] As $z$ tends to 0, the behavior of $\Phi_{\mathrm{Be}}(z)$ is
\begin{equation}\label{local behavior near 0 of P_Be}
\begin{array}{l l}
\displaystyle \Phi_{\mathrm{Be}}(z) = \left\{ \begin{array}{l l}
\begin{pmatrix}
\bigO(1) & \bigO(\log z) \\
\bigO(1) & \bigO(\log z) 
\end{pmatrix}, & |\arg z| < \frac{2\pi}{3}, \\
\begin{pmatrix}
\bigO(\log z) & \bigO(\log z) \\
\bigO(\log z) & \bigO(\log z) 
\end{pmatrix}, & \frac{2\pi}{3}< |\arg z| < \pi.
\end{array}  \right.
\end{array}
\end{equation}
\end{itemize}
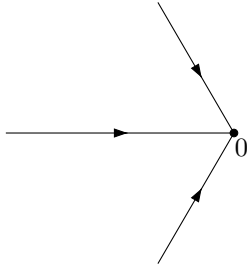
\begin{figure}
\centering
\begin{tikzpicture}
\draw (-3,0)--(0,0);
\draw (0,0)--(120:2);
\draw (0,0)--(-120:2);
\node at (0.1,-0.2) {$0$};
\draw[fill] (0,0) circle (0.05);
\draw[black,arrows={-Triangle[length=0.18cm,width=0.12cm]}]
(180:1.4) --  ++(0:0.001);
\draw[black,arrows={-Triangle[length=0.18cm,width=0.12cm]}]
(120:0.85) --  ++(-60:0.001);
\draw[black,arrows={-Triangle[length=0.18cm,width=0.12cm]}]
(-120:0.85) --  ++(60:0.001);
\end{tikzpicture}
\caption{\label{figBessel}The jump contour $\Sigma_{\mathrm{Be}}$ for $\Phi_{\mathrm{Be}}$.}
\end{figure}
The unique solution of this RH problem exists and can be explicitly constructed in terms of Bessel functions \cite{DIZ,KMcLVAV}.

\end{document}